\documentclass[11pt,a4paper]{amsart}
\usepackage{latexsym}
\usepackage{graphicx}
\usepackage{subfig}
\usepackage{caption}
\usepackage{float}
\usepackage{enumerate}
\usepackage[top=3.2cm,bottom=3.8cm,left=3cm,right=2cm]{geometry}
\usepackage{mathrsfs}
\usepackage{amssymb}
\usepackage{amsbsy}
\usepackage{mathtools}
\usepackage{xcolor}
\usepackage{needspace}
\usepackage[colorlinks,linkcolor=blue,citecolor=blue,pagebackref]{hyperref}

\newtheorem{theorem}{Theorem}[section]
\newtheorem{lemma}[theorem]{Lemma}

\theoremstyle{definition}

\theoremstyle{remark}

\numberwithin{equation}{section}

\newcommand{\R}{\mathbb{R}}
\newcommand{\Sph}{\mathbb{S}}
\newcommand{\eps}{\varepsilon}
\newcommand{\Var}{\operatorname{Var}}

\newcommand{\Hess}{\nabla^2}
\newcommand{\HS}{\mathrm{HS}}
\newcommand{\Id}{\mathrm{Id}}
\newcommand{\E}{\mathbb{E}}
\newcommand{\II}{\mathrm{II}_K}
\newcommand{\tr}{\operatorname{tr}}

\makeatletter
\renewcommand{\section}{\@startsection{section}{1}%
	\z@{.7\linespacing\@plus\linespacing}{.5\baselineskip}%
	{\normalfont\scshape\centering}}
\renewcommand{\subsection}{\@startsection{subsection}{2}%
	\z@{.5\linespacing\@plus.7\linespacing}{.5\baselineskip}%
	{\normalfont\bfseries}}
\renewcommand{\subsubsection}{\@startsection{subsubsection}{3}%
	\z@{.5\linespacing\@plus.7\linespacing}{.5\baselineskip}%
	{\normalfont\itshape}}
\makeatother

\newcommand{\proofstep}[2]{%
	\par\Needspace{5\baselineskip}%
	\addvspace{\smallskipamount}%
	\noindent\textbf{Step #1.} #2\par
	\nobreak\vspace{.0\baselineskip}\nobreak
}
\begin{document}
	
	\begin{center}
		{\large\bf The Brunn--Minkowski inequality for the Gaussian measure}
	\end{center}
	
	\vskip 15pt
	\begin{center}
		{\small\bf Kai-Wen Yang}\\~~ \\
		\small{School of Mathematical Sciences, Key Laboratory of Intelligent Computing and Applications (Ministry of Education), Tongji University, Shanghai, 200092, China}
	\end{center}
	
	\vskip 5pt
	\begin{NoHyper}
		\footnotetext{E-mail address:  yangkaiwen@tongji.edu.cn.}
	\end{NoHyper}
	
	\begin{center}
		\begin{minipage}{14cm}
			{{\bf Abstract:} Let $\gamma_n$ be the standard Gaussian measure on $\R^n$, $n\ge2$, and let $\alpha_\gamma(n)$ be the largest number for which
				\[
				\gamma_n(\lambda K+(1-\lambda)L)^{\alpha_\gamma(n)} \ge \lambda\gamma_n(K)^{\alpha_\gamma(n)} +(1-\lambda)\gamma_n(L)^{\alpha_\gamma(n)}
				\]
				holds for all convex bodies $K,L\subset\R^n$ containing the origin and all $\lambda\in[0,1]$.
				In this paper, we prove that
				\[
				\alpha_\gamma(n) =1-\frac{2}{n-1} \frac{\Gamma(\frac n2)^2}{\Gamma(\frac{n-1}{2})^2}.
				\]
				The core of the proof is a raywise radial--tangential localization of the Hessian energy of the solution of a Neumann problem, which reduces source selection of the Neumann problem to a one-dimensional optimization.
				Monotonicity in the segment length and Laguerre spectral analysis determine the sharp one-dimensional value, whereas the planar endpoint is treated separately.
			}
			
			
		\end{minipage}
	\end{center}
	
	\vskip 20pt
	
	\section{\bf Introduction}
	
	The setting for this paper is the $n$-dimensional Euclidean space, $\mathbb{R}^n$. Let $V_n$ be the $n$-dimensional Lebesgue measure. The classical Brunn--Minkowski inequality asserts that for all nonempty Borel sets $K, L \subseteq \mathbb{R}^n$ and all $\lambda \in [0,1]$, 
	$$
	V_n(\lambda K+(1-\lambda) L)^\frac{1}{n} \geq \lambda V_n(K)^\frac{1}{n}+ (1-\lambda) V_n(L)^\frac{1}{n}.
	$$
	Here, $$\lambda K+(1-\lambda) L=\left\{\lambda x+(1-\lambda) y: x \in K, y \in L\right\}$$ is the Minkowski combination of $K$ and $L$. 
	
	The Brunn--Minkowski inequality is one of the cornerstones of convex geometry. See, e.g., books by Gardner \cite{Gardner1} and Schneider \cite{Sch}.  
	It is also closely related to
	many other fundamental inequalities, such as the isoperimetric inequality, the Pr\'ekopa--Leindler inequality, the Sobolev inequality and the Brascamp--Lieb inequality. See, e.g., Barthe \cite{Barthe} and Bobkov and Ledoux \cite{Bobkov1,Bobkov2}. We refer to the survey by Gardner \cite{gar1} for more.
	
	In 2010, Gardner and Zvavitch \cite{Gardner} asked whether the standard Gaussian distribution $\gamma_n$, i.e., the measure with density $\frac{d \gamma_n}{d x}=\frac{1}{(2 \pi)^{n / 2}} e^{-\frac{|x|^2}{2}}$, satisfies
	\[
	\gamma_n(\lambda K+(1-\lambda)L)^{\frac{1}{n}} \ge \lambda\gamma_n(K)^{\frac{1}{n}} +(1-\lambda)\gamma_n(L)^{\frac{1}{n}}
	\] 
	for all $\lambda\in (0,1)$ and all convex bodies $K$, $L$ containing the origin? 
	In 2013, a counterexample to this problem was constructed by Nayar and Tkocz  \cite{Nayar}. 
	
	In 2021,  Kolesnikov and Livshyts \cite{Livshyts} proved that if $K$ and $L$ are  convex bodies containing the origin in $\mathbb{R}^n$ and  $\lambda \in[0,1]$, then
	$$
	\gamma_n(\lambda K+(1-\lambda) L)^{\frac{1}{2n}} \geq \lambda \gamma_n(K)^{\frac{1}{2n}}+(1-\lambda) \gamma_n(L)^{\frac{1}{2n}} .$$ 
	Later, Eskenazis and Moschidis \cite{Moschidis} settled the Gardner-Zvavitch problem when  $K$ and $L$ are \emph{origin-symmetric} convex bodies: If $K$ and $L$ are origin-symmetric convex bodies in $\mathbb{R}^n$ and  $\lambda \in[0,1]$, then
	\begin{equation}\label{gaussSys}
		\gamma_n(\lambda K+(1-\lambda) L)^{\frac{1}{n}} \geq \lambda \gamma_n(K)^{\frac{1}{n}}+(1-\lambda) \gamma_n(L)^{\frac{1}{n}} .
	\end{equation}
	In 2025, Aishwarya and Li \cite{Aishwarya} gave an optimal-transport proof of \eqref{gaussSys} and established related entropic and functional inequalities.
	Aishwarya and Rotem \cite[Theorem~1.3]{rotem} extended the Gaussian $\frac1{2n}$-concavity inequality to star bodies with respect to the origin.
	
	For $n\ge2$, let $\mathcal K_o^n$ be the class of convex bodies in $\R^n$ containing the origin and define
	$\alpha_\gamma(n)$ as the largest number for which
	\[
	\gamma_n(\lambda K+(1-\lambda)L)^{\alpha_\gamma(n)} \ge \lambda\gamma_n(K)^{\alpha_\gamma(n)} +(1-\lambda)\gamma_n(L)^{\alpha_\gamma(n)}
	\]
	holds for all convex bodies $K,L\in\mathcal K_o^n$ and all $\lambda\in[0,1]$. Recently, 
	Xiong and Yang \cite[Theorem~1.1]{XiongYang2026} gave that
	\[
	\frac1n e^{\frac n2}\Bigl(\frac n2\Bigr)^{\frac n2}\Gamma\Bigl(1-\frac n2,\frac n2\Bigr)
	\le\alpha_\gamma(n)\le 1-\frac{2}{n-1}\frac{\Gamma(\frac n2)^2}{\Gamma(\frac{n-1}{2})^2},
	\]
	where $\Gamma(a,x)=\int_x^\infty t^{a-1}e^{-t}\,dt$ is the upper incomplete gamma function. In this paper, we further determine  $\alpha_\gamma(n)$. 
	
	\begin{theorem}\label{thm:main}
		Let $\alpha_\gamma(n)$ be the largest number for which
		\[
		\gamma_n(\lambda K+(1-\lambda)L)^{\alpha_\gamma(n)} \ge \lambda\gamma_n(K)^{\alpha_\gamma(n)} +(1-\lambda)\gamma_n(L)^{\alpha_\gamma(n)}
		\]
		holds for all convex bodies $K,L\subset\R^n$ containing the origin and all $\lambda\in[0,1]$. Then
		\[
		\alpha_\gamma(n)=q_n:=1-\frac{2}{n-1}\frac{\Gamma(\frac n2)^2}{\Gamma(\frac{n-1}{2})^2}.
		\]
	\end{theorem}
	
	For the lower bound, we follow the local-to-global route of Kolesnikov and Livshyts \cite{Livshyts} in a streamlined and more flexible form; see Section~\ref{sec:local-reduction}.
	A fixed-source version of this framework was used in \cite{XiongYang2026} for their lower bound, whereas here the compatible source in $Lu=g$, with $L=\Delta-\langle x,\nabla\rangle$, may depend on the boundary datum.
	
	The core idea of the proof is to localize this source-selection problem ray by ray.
	Put $m=n-1$.
	At $x=r\theta$, the radial direction is normal to the sphere $r\Sph^{n-1}$, and the remaining $m$ directions form the tangential block of $\Hess u$.
	Since $Lu=g$, applying Cauchy--Schwarz only to the trace of this block gives
	\[
	\|\Hess u\|_{\HS}^2+|\nabla u|^2
	\ge u_{rr}^2+u_r^2+\frac{(g+ru_r-u_{rr})^2}{m}.
	\]
	See Lemma \eqref{lem:polar-hessian}.
	Polar integration now turns the original $n$-dimensional local estimate into a family of independent weighted problems on the segments $(0,\rho_K(\theta))$.
	In particular, the freedom to choose a source in $K$ becomes a one-dimensional source optimization on each ray.
	This raywise localization is therefore the step that exposes the sharp constant.
	
	This geometric localization principle is not intrinsically Gaussian.
	For another smooth rotationally invariant measure, polar disintegration still separates the rays, while the radial density changes the local energy and the resulting one-dimensional operator.
	Thus, once an appropriate local Neumann--Reilly criterion is available, the same mechanism should lead to measure-dependent one-dimensional source optimizations.
	This suggests a route beyond the Gaussian setting, although the sharp spectral computation below is Gaussian-specific.
	
	For the Gaussian measure, the resulting one-dimensional problem on a segment of length $R$ has max--min value $\lambda_m(R)$. See definition for \eqref{eq:lambda-definition}. 
	For $m>1$, we identify its saddle point, prove that $\lambda_m(R)$ is nonincreasing in $R$, and evaluate its half-line limit by Laguerre spectral analysis:
	\[
	\lim_{R\to\infty}\lambda_m(R)=1-\frac2m\frac{\Gamma(\frac{m+1}{2})^2}{\Gamma(\frac m2)^2}.
	\]
	The degenerate endpoint $m=1$ is treated separately.
	Finally, measurable maximizing profiles are assembled along the rays of $K$, giving the sharp local estimate and hence the lower bound.
	The matching upper bound was already obtained by Xiong and Yang \cite[Theorem~1.1]{XiongYang2026}; see Section~\ref{sec:upper-bound} for details.
	
	The paper is organized as follows.
	Section~\ref{sec:prelim} collects some facts will be used later.
	Section~\ref{sec:local-reduction} derives a lower bound for $\alpha_\gamma(n)$ by refining the analytic route of Kolesnikov and Livshyts \cite{Livshyts}.
	Section~\ref{sec:upper-bound} recalls the matching upper bound of Xiong and Yang.
	Section~\ref{sec:one-dim} solves the one-dimensional variational problem.
	Section~\ref{lift} constructs the source ray by ray and proves Theorem~\ref{thm:main}.
	Appendix~\ref{app:explicit-profiles} records closed formulas for the maximizing profiles; these formulas are not needed for the proof of the theorem.
	
	\vskip3pt \noindent{\bf Declaration on the use of AI}: The solution of the one-dimensional variational problem in Section 5 and the appendix are revised versions of arguments generated by GPT 5.6 Sol and subsequently checked and rewritten by the author.
	GPT 5.6 Sol was also used to identify a gap in the proof of Lemma \ref{lem:profile-dependence} in an earlier version of the manuscript. GPT-5.6 Sol was then used to polish the writing.
	
	\vskip 20pt
	
	\section{\bf Preliminaries}\label{sec:prelim}
	
	Throughout the paper $n\ge2$.
	In Subsection~\ref{ssec:radial-tools} and Section~\ref{sec:one-dim}, $m\ge1$ is an independent real parameter; elsewhere $m=n-1$.
	
	We write $\langle\cdot,\cdot\rangle$ and $|\cdot|$ for the Euclidean inner product and norm.
	The Euclidean Hessian is denoted by $\Hess u$, and $\|\cdot\|_{\HS}$ is the Hilbert--Schmidt norm.
	The unit Euclidean ball and sphere are denoted by $B_2^n$ and $\Sph^{n-1}$.
	The spherical Hessian is denoted by $\nabla_{\Sph^{n-1}}^2$.
	We write $\mathcal H^k$ for $k$-dimensional Hausdorff measure and $\Id$ for the identity on the relevant tangent space.
	
	\subsection{Convex bodies and Gaussian variations}
	
	A convex body in $\R^n$ is a compact convex set with nonempty interior.
	Its support function is
	\[
	h_K(\theta)=\max_{x\in K}\langle x,\theta\rangle, \qquad \theta\in\Sph^{n-1}.
	\]
	For convex bodies $K,L$ and $\alpha,\beta\ge0$,
	\[
	h_{\alpha K+\beta L}=\alpha h_K+\beta h_L.
	\]
	
	A convex body belongs to $\mathcal C^2_+$ if its boundary is a $C^2$ hypersurface with everywhere positive Gauss curvature.
	We write $\mathcal K_{+,o}^n$ for the members of $\mathcal C^2_+$ containing the origin in their interiors.
	For $h\in C^2(\Sph^{n-1})$, its spherical curvature matrix is
	\[
	Q(h)=\nabla_{\Sph^{n-1}}^2h+h\Id \quad\text{on }T\Sph^{n-1}.
	\]
	A convex body $K$ belongs to $\mathcal C^2_+$ precisely when $h_K\in C^2(\Sph^{n-1})$ and $Q(h_K)$ is positive definite; moreover, $0\in\operatorname{int}K$ precisely when $\min_{\Sph^{n-1}}h_K>0$.
	These standard support-function facts may be found in \cite{Sch}.
	
	If $0\in\operatorname{int}K$, its radial function is
	\[
	\rho_K(\theta)=\max\{r\ge0:r\theta\in K\}, \qquad \theta\in\Sph^{n-1}.
	\]
	It is continuous and strictly positive.
	Equivalently, the inclusions $aB_2^n\subset K\subset bB_2^n$ for some $0<a\le b<\infty$ give $a\le\rho_K\le b$ on the sphere.
	
	Let $K\in\mathcal K_{+,o}^n$ and let $\nu_K:\partial K\to\Sph^{n-1}$ be the Gauss map.
	It is a $C^1$ diffeomorphism, and
	\[
	x_K(\theta)=\nabla_{\Sph^{n-1}}h_K(\theta)+h_K(\theta)\theta
	\]
	is its inverse parametrization; in these coordinates $Q(h_K)$ is the inverse Weingarten map.
	Thus, if $K,L\in\mathcal K_{+,o}^n$, then $Q((1-t)h_K+th_L)>0$ for $0\le t\le1$.
	Likewise, for $\psi\in C^2(\Sph^{n-1})$, compactness gives both $Q(h_K+s\psi)>0$ and $\min(h_K+s\psi)>0$ when $|s|$ is small.
	Hence $h_K+s\psi$ is the support function of a body $K_s\in\mathcal K_{+,o}^n$.
	
	Let $\II(X,Y)=\langle D_X\nu_K,Y\rangle$ be the second fundamental form with respect to the outer unit normal.
	With this convention $\II$ is positive definite.
	By compactness, $\II\ge c_K\Id$ on $\partial K$ for some $c_K>0$, so $\II^{-1}$ is continuous and bounded.
	We write $\nabla_{\partial K}$ for the tangential gradient and put
	\[
	d\gamma_{\partial K} =(2\pi)^{-\frac n2}e^{-\frac{|x|^2}{2}}\,d\mathcal H^{n-1}(x), \qquad H_\gamma=\tr(\II)-\langle x,\nu_K(x)\rangle.
	\]
	Set $f=\psi\circ\nu_K\in C^1(\partial K)$.
	The following variational formulas were obtained by Kolesnikov and Milman \cite[proof of Theorem~6.6]{milman1}.
	\begin{equation}\label{eq:var}
		\frac{d}{ds}\gamma_n(K_s)\Big|_{s=0}=\int_{\partial K}f\,d\gamma_{\partial K}, \qquad \frac{d^2}{ds^2}\gamma_n(K_s)\Big|_{s=0}=\int_{\partial K}\Bigl(H_\gamma f^2-\bigl\langle\II^{-1}\nabla_{\partial K}f,\nabla_{\partial K}f\bigr\rangle\Bigr)d\gamma_{\partial K}.
	\end{equation}
	These formulas generalize the formulas obtained by Colesanti \cite{Colesanti} in the Euclidean space.
	
	For later use, if $f\in C^1(\partial K)$, define the boundary quadratic form
	\[
	\mathcal B_K(f) =\int_{\partial K} \Bigl( H_\gamma f^2 -\bigl\langle\II^{-1}\nabla_{\partial K}f, \nabla_{\partial K}f\bigr\rangle \Bigr)d\gamma_{\partial K}.
	\]
	
	\subsection{Functional analytic tools}
	\label{ssec:functional-tools}
	
	For an interval $I\subset(0,\infty)$ and a positive weight $w$ that is bounded above and away from zero on every compact subinterval of $I$, the space $H^1(I,w\,dr)$ consists of the functions in $H^1_{\mathrm{loc}}(I)$ for which
	\[
	\int_I (|v|^2+|v'|^2)w\,dr<\infty.
	\]
	It is a Hilbert space with the evident norm.
	We shall repeatedly use the following standard theorem; see \cite[Corollary~5.8]{Brezis2011}.
	
	\begin{theorem}[Lax--Milgram]
		\label{thm:lax-milgram}
		Let $X$ be a real Hilbert space and let $a:X\times X\to\R$ be bilinear.
		If there are constants $M<\infty$ and $c>0$ such that
		\[
		|a(u,v)|\le M\|u\|_X\|v\|_X, \qquad a(v,v)\ge c\|v\|_X^2 \quad(u,v\in X),
		\]
		then, for every $\ell\in X^*$, there is a unique $u\in X$ satisfying
		\[
		a(u,v)=\ell(v)\quad(v\in X), \qquad \|u\|_X\le c^{-1}\|\ell\|_{X^*}.
		\]
		If $a$ is symmetric, this $u$ is also the unique minimizer of
		\[
		v\longmapsto \frac12a(v,v)-\ell(v).
		\]
	\end{theorem}
	
	We next recall the operator terminology needed below.
	A densely defined operator $A$ on a real Hilbert space $\mathcal H$ is self-adjoint if $A=A^*$, including equality of their domains, and it is nonnegative if $\langle Au,u\rangle_{\mathcal H}\ge0$ for every $u\in\operatorname{Dom}(A)$.
	If $D:\operatorname{Dom}(D)\subset\mathcal H\to\mathcal H$ is densely defined, its Hilbert-space adjoint is defined by
	\[
	\operatorname{Dom}(D^*) =\Bigl\{g\in\mathcal H: h\longmapsto\langle Dh,g\rangle_{\mathcal H} \text{ is bounded in the $\mathcal H$-norm}\Bigr\}.
	\]
	For $g$ in this domain, the bounded functional extends uniquely to $\mathcal H$, and the Riesz representation theorem determines the unique vector $D^*g$ satisfying
	\[
	\langle Dh,g\rangle_{\mathcal H} =\langle h,D^*g\rangle_{\mathcal H} \qquad(h\in\operatorname{Dom}(D)).
	\]
	The adjoint is closed; if $D$ is closed, then $D^*$ is also densely defined.
	
	A densely defined nonnegative symmetric form $\mathfrak q$ is closed if its domain is complete for $\bigl(\|u\|_{\mathcal H}^2+\mathfrak q(u,u)\bigr)^{\frac12}$.
	An operator $A$ is called associated with $\mathfrak q$ if, for $u\in\operatorname{Dom}(\mathfrak q)$ and $f\in\mathcal H$,
	\[
	u\in\operatorname{Dom}(A),\quad Au=f \quad\Longleftrightarrow\quad \mathfrak q(u,h)=\langle f,h\rangle_{\mathcal H} \quad(h\in\operatorname{Dom}(\mathfrak q)).
	\]
	The following existence theorem fixes the self-adjoint realization used on the half-line.
	
	\begin{theorem}[Self-adjoint realization]
		\label{thm:self-adjoint-realization}
		Let $\mathcal H$ be a real Hilbert space and let $D:\operatorname{Dom}(D)\subset\mathcal H\to\mathcal H$ be densely defined and closed.
		Define the form
		\[
		\mathfrak q(u,h)=\langle Du,Dh\rangle_{\mathcal H}, \qquad \operatorname{Dom}(\mathfrak q)=\operatorname{Dom}(D).
		\]
		It is densely defined, nonnegative, and closed, and it has a unique associated nonnegative self-adjoint operator.
		This operator is $D^*D$, with
		\[
		\operatorname{Dom}(D^*D) =\{u\in\operatorname{Dom}(D):Du\in\operatorname{Dom}(D^*)\}.
		\]
		For every $\lambda>0$ and $f\in\mathcal H$, there is a unique $u\in\operatorname{Dom}(D)$ satisfying
		\[
		\mathfrak q(u,h)+\lambda\langle u,h\rangle_{\mathcal H} =\langle f,h\rangle_{\mathcal H} \quad(h\in\operatorname{Dom}(D))
		\]
		and this solution belongs to $\operatorname{Dom}(D^*D)$ and equals $(D^*D+\lambda)^{-1}f$.
	\end{theorem}
	
	\begin{proof}
		The graph norm makes $\operatorname{Dom}(D)$ complete, so $\mathfrak q$ is closed.
		The representation theorem gives a unique associated nonnegative self-adjoint operator.
		Its defining identity says precisely that $Du\in\operatorname{Dom}(D^*)$ and that the operator acts as $D^*Du$; hence it is $D^*D$.
		See \cite[Chapter~VI, Theorems~2.1 and~2.6, and Example~2.13]{Kato1995}.
		
		For $\lambda>0$, the shifted form is continuous and coercive in the graph norm because
		\[
		\|Du\|_{\mathcal H}^2+\lambda\|u\|_{\mathcal H}^2 \ge \min\{1,\lambda\} \bigl(\|Du\|_{\mathcal H}^2+\|u\|_{\mathcal H}^2\bigr).
		\]
		Theorem~\ref{thm:lax-milgram} gives a unique weak solution.
		Rewriting its identity as $\mathfrak q(u,h)=\langle f-\lambda u,h\rangle_{\mathcal H}$ shows, by the definition of the associated operator, that $u\in\operatorname{Dom}(D^*D)$ and $D^*Du=f-\lambda u$.
	\end{proof}
	
	Thus Lax--Milgram constructs the shifted weak solution, while the form representation theorem supplies the self-adjoint operator and identifies that solution with its resolvent.
	When the operator has a complete orthogonal eigenbasis, the spectral theorem diagonalizes the resolvent.
	
	\subsection{The Gaussian Neumann problem and the weighted Reilly formula}
	
	Fix $K\in\mathcal K_{+,o}^n$.
	Throughout this subsection the Sobolev spaces on $K$ and $\partial K$ are the standard unweighted spaces.
	Since $K$ is compact, the Gaussian density is smooth and bounded above and below by positive constants on $K$; the corresponding weighted and unweighted interior and boundary norms are equivalent.
	
	The Ornstein--Uhlenbeck operator is
	\[
	Lu=\Delta u-x\cdot\nabla u =e^{\frac{|x|^2}{2}}\operatorname{div}(e^{-\frac{|x|^2}{2}}\nabla u).
	\]
	Consequently, for smooth $u$ and $\varphi$,
	\[
	\int_K(Lu)\varphi\,d\gamma_n =-\int_K\langle\nabla u,\nabla\varphi\rangle\,d\gamma_n +\int_{\partial K}u_{\nu_K}\varphi\,d\gamma_{\partial K}.
	\]
	
	Let $g\in L^2(K,\gamma_n)$ and $f\in H^{\frac12}(\partial K)$.
	Consider
	\begin{equation}\label{eq:neumann}
		\begin{cases}
			Lu=g & \text{in }K,\\
			u_{\nu_K}=f & \text{on }\partial K.
		\end{cases}
	\end{equation}
	Here $u_{\nu_K}=\langle\nabla u,\nu_K\rangle$.
	A function $u\in H^1(K)$ is a weak solution of \eqref{eq:neumann} if
	\begin{equation}\label{eq:weak-neumann}
		\int_K\langle\nabla u,\nabla\varphi\rangle d\gamma_n =\int_{\partial K}f\varphi\,d\gamma_{\partial K} -\int_K g\varphi\,d\gamma_n, \qquad \varphi\in H^1(K).
	\end{equation}
	The associated compatibility condition is
	\begin{equation}\label{eq:compatibility}
		\int_K g\,d\gamma_n =\int_{\partial K}f\,d\gamma_{\partial K}.
	\end{equation}
	Put $\rho(x)=e^{-\frac{|x|^2}{2}}$.
	Since $\rho I$ is uniformly positive on $K$, the variational Neumann theorem for strongly elliptic operators \cite[Theorem~4.11]{McLean2000}, applied to
	\[
	-\operatorname{div}(\rho\nabla u)=-\rho g\quad\text{in }K,\qquad \rho\,\partial_{\nu_K}u=\rho f\quad\text{on }\partial K,
	\]
	shows that \eqref{eq:compatibility} is sufficient, that the solution is unique modulo constants, and that its Gaussian mean-zero representative satisfies
	\begin{equation}\label{eq:H1-estimate}
		\|u\|_{H^1(K)}\le C_K\bigl(\|g\|_{L^2(K,\gamma_n)}+\|f\|_{H^{-\frac12}(\partial K)}\bigr).
	\end{equation}
	
	Following \cite[\S2.2]{milman1}, let $\mathcal S_N(K)$ denote the class of functions
	\[
	u\in C^2_{\mathrm{loc}}(K^\circ)\cap C^1(K), \qquad u_{\nu_K}\in C^1(\partial K).
	\]
	This is the natural class for the generalized Reilly formula on a $C^2$ domain.
	
	We shall use the following standard elliptic regularity result.
	If $\Omega\subset\R^n$ is a bounded $C^2$ domain, $F\in L^2(\Omega)$, $b\in H^{\frac12}(\partial\Omega)$, and $U\in H^1(\Omega)$ is a weak solution of
	\[
	\Delta U=F\quad\text{in }\Omega, \qquad \partial_\nu U=b\quad\text{on }\partial\Omega,
	\]
	then $U\in H^2(\Omega)$ and
	\[
	\|U\|_{H^2(\Omega)} \le C_\Omega\bigl( \|F\|_{L^2(\Omega)}+ \|b\|_{H^{\frac12}(\partial\Omega)}+ \|U\|_{L^2(\Omega)}\bigr).
	\]
	This follows from the shifted Neumann estimate \cite[Corollary~2.2.2.6]{Grisvard1985}, applied to $-\Delta U+U=-F+U$.
	The $L^2$ term is necessary because Neumann solutions are defined only modulo constants.
	
	\begin{lemma}
		\label{lem:gaussian-neumann}
		Problem \eqref{eq:neumann} has a weak solution if and only if \eqref{eq:compatibility} holds.
		Whenever it exists, the solution is unique modulo constants.
		Its Gaussian mean-zero representative belongs to $H^2(K)$ and satisfies
		\begin{equation}\label{eq:H2-estimate}
			\|u\|_{H^2(K)} \le C_K\bigl(\|g\|_{L^2(K,\gamma_n)} +\|f\|_{H^{\frac12}(\partial K)}\bigr).
		\end{equation}
		If, in addition, $g\in C^\alpha(K)$ for some $\alpha\in(0,1)$ and $f\in C^1(\partial K)$, then $u\in\mathcal S_N(K)$.
	\end{lemma}
	
	\begin{proof}
		\proofstep{1}{Variational solvability.}

		Taking $\varphi=1$ in \eqref{eq:weak-neumann} proves the necessity of \eqref{eq:compatibility}.
		Conversely, the variational Neumann result following \eqref{eq:compatibility} gives existence, uniqueness modulo constants, and the estimate \eqref{eq:H1-estimate}.
		
		\proofstep{2}{The $H^2$ estimate.}

		For every $\psi\in H^1(K)$, the function $\rho^{-1}\psi$ belongs to $H^1(K)$ and
		\[
		\rho\nabla(\rho^{-1}\psi)=\nabla\psi+x\psi.
		\]
		Using $\rho^{-1}\psi$ in \eqref{eq:weak-neumann} and cancelling the Gaussian normalizing constant gives
		\[
		\int_K\langle\nabla u,\nabla\psi\rangle\,dx=\int_{\partial K}f\psi\,d\mathcal H^{n-1}-\int_K(g+x\cdot\nabla u)\psi\,dx.
		\]
		Therefore $u$ is a weak solution of the standard Neumann problem with $F=g+x\cdot\nabla u\in L^2(K)$ and boundary datum $f$.
		The equivalence of the weighted and unweighted norms and the estimate from Step~1 give
		\[
		\|F\|_{L^2(K)}+\|u\|_{L^2(K)}\le C_K\bigl(\|g\|_{L^2(K,\gamma_n)}+\|f\|_{H^{-\frac12}(\partial K)}\bigr).
		\]
		The standard Neumann estimate stated before the lemma and the embedding $H^{\frac12}(\partial K)\hookrightarrow H^{-\frac12}(\partial K)$ now yield
		\[
		\|u\|_{H^2(K)}\le C_K\bigl(\|F\|_{L^2(K)}+\|f\|_{H^{\frac12}(\partial K)}+\|u\|_{L^2(K)}\bigr)\le C_K\bigl(\|g\|_{L^2(K,\gamma_n)}+\|f\|_{H^{\frac12}(\partial K)}\bigr),
		\]
		which is \eqref{eq:H2-estimate}.
		
		\proofstep{3}{Smooth data.}

		If $g\in C^\alpha(K)$ and $f\in C^1(\partial K)$, \cite[Theorem~2.5]{milman1} supplies a solution $\widetilde u\in\mathcal S_N(K)$ of the same compatible Neumann problem.
		It is also a weak solution, so Step~1 shows that $u-\widetilde u$ is constant.
		Since adding a constant preserves $\mathcal S_N(K)$, the mean-zero solution $u$ belongs to $\mathcal S_N(K)$.
	\end{proof}
	
	For later use, we recall the generalized weighted Reilly formula.
	If $u\in\mathcal S_N(K)$, then
	\begin{align}
		\int_K(Lu)^2\,d\gamma_n ={}&\int_K\bigl(\|\Hess u\|_{\HS}^2+|\nabla u|^2\bigr)d\gamma_n \notag\\
		&+\int_{\partial K}\bigl(H_\gamma u_{\nu_K}^2-2\langle\nabla_{\partial K}u,\nabla_{\partial K}u_{\nu_K}\rangle+\langle\II\nabla_{\partial K}u,\nabla_{\partial K}u\rangle\bigr)d\gamma_{\partial K}.
		\label{eq:reilly}
	\end{align}
	This is the Gaussian specialization of \cite[Theorem~1.1]{milman2}; see also \cite[Theorem~2.1]{milman1}.
	
	\subsection{The radial half-line operator and its Laguerre resolution}
	\label{ssec:radial-tools}
	
	This subsection fixes the finite-energy self-adjoint realization at the singular endpoint $r=0$ and then diagonalizes it by Laguerre polynomials.
	Formal adjoints on finite intervals, including the boundary term at $r=R$, are introduced only in Section~\ref{sec:one-dim}.
	
	\subsubsection{The half-line realization.}
	
	Let $w_m(r)=r^m e^{-\frac{r^2}{2}}$, where $m\ge1$. Define the derivative operator in $L^2((0,\infty),w_m\,dr)$ by
	\[
	Dv=v', \qquad \operatorname{Dom}(D)=H^1((0,\infty),w_m\,dr).
	\]
	The operator is densely defined because $C_c^\infty(0,\infty)$ is dense in the weighted $L^2$ space, and it is closed because its graph norm is the weighted $H^1$ norm.
	Theorem~\ref{thm:self-adjoint-realization} therefore gives the nonnegative self-adjoint operator $D^*D$ generated by
	\[
	(v,h)\longmapsto\int_0^\infty v'h'w_m\,dr, \qquad v,h\in H^1((0,\infty),w_m\,dr).
	\]
	Here $D^*$ is the Hilbert-space adjoint.
	In particular, for every $f\in L^2((0,\infty),w_m\,dr)$, the resolvent $v=(D^*D+2)^{-1}f$ is the unique weighted $H^1$ solution of
	\[
	\int_0^\infty(v'h'+2vh)w_m\,dr =\int_0^\infty fhw_m\,dr \quad\bigl(h\in H^1((0,\infty),w_m\,dr)\bigr).
	\]
	The shift makes this form coercive even though constants lie in $\ker D$.
	
	For $v\in\operatorname{Dom}(D^*D)$, testing against $C_c^\infty(0,\infty)$ and integrating by parts gives the interior expression
	\[
	D^*D v=-v''-\frac{w_m'}{w_m}v' =-v''+\Bigl(r-\frac mr\Bigr)v'
	\]
	in the sense of distributions.
	
	The differential expression above describes only the interior action of $D^*D$.
	The behavior at the singular endpoint $0$ is determined by the finite-energy space
	\[
	V_m:=H^1((0,\infty),w_m\,dr),\qquad
	\|v\|_{V_m}^2:=\int_0^\infty(|v|^2+|v'|^2)w_m\,dr.
	\]
	We record two consequences of this choice.
	
	\noindent\textbf{First}. Functions in $V_m$ can be approximated by functions that vanish near $0$.
	
	Since $w_m(r)^{-1}\sim r^{-m}$ and $m\ge1$,
	\[
	\int_0^b w_m^{-1}\,dr=\infty
	\]
	for every $b>0$.
	For $0<\eps<b$, define $\chi_\eps=0$ on $(0,\eps]$, $\chi_\eps=1$ on $[b,\infty)$, and
	\[
	\chi_\eps'(r)=\frac{w_m(r)^{-1}}{\int_\eps^b w_m(t)^{-1}\,dt}.
	\]
	For every fixed $r\in(0,b)$,
	\[
	1-\chi_\eps(r)
	=\frac{\int_r^b w_m(t)^{-1}\,dt}{\int_\eps^b w_m(t)^{-1}\,dt}
	\longrightarrow0,
	\]
	and
	\[
	\int_0^\infty|\chi_\eps'|^2w_m\,dr=\Bigl(\int_\eps^b w_m^{-1}\,dr\Bigr)^{-1}\longrightarrow0.
	\]
	If $v\in V_m\cap L^\infty$, dominated convergence gives
	\[
	\|(\chi_\eps-1)v\|_{L^2(w_m\,dr)}
	+\|(\chi_\eps-1)v'\|_{L^2(w_m\,dr)}\longrightarrow0,
	\]
	while
	\[
	\|\chi_\eps'v\|_{L^2(w_m\,dr)}
	\le\|v\|_\infty
	\Bigl(\int_\eps^b w_m^{-1}\,dr\Bigr)^{-\frac12}
	\longrightarrow0.
	\]
	The product rule therefore gives $\chi_\eps v\to v$ in $V_m$.
	Lipschitz truncation reduces a general $v\in V_m$ to the bounded case.
	A cutoff at infinity and mollification then give the required approximation by functions in $C_c^\infty(0,\infty)$.
	
	Thus $C_c^\infty(0,\infty)$ is dense in $V_m$ in the form norm; this is the form-core property used below.
	On a finite interval $(0,R)$, choose $b<R$.
	Then $\chi_\eps v$ vanishes near $0$ but agrees with $v$ near $R$.
	Passing to the limit in a weak integration by parts therefore removes the boundary term at $0$ and retains the one at $R$.
	
	\noindent\textbf{Second.} Finite energy selects the regular ODE branch at $0$.
	
	Substituting $v(r)=r^\alpha$ into the most singular part
	\[
	-v''-\frac mr v'
	\]
	of $(D^*D+2)v=0$ gives the indicial equation
	\[
	\alpha(\alpha+m-1)=0.
	\]
	Thus the indicial roots are $0$ and $1-m$, with a double root when $m=1$.
	The root $0$ gives a bounded regular solution, while a linearly independent solution has leading behavior $r^{1-m}$ when $m>1$ and $\log r$ when $m=1$.
	For the latter solution,
	\[
	|v'(r)|^2w_m(r)\asymp r^{-m},
	\]
	and hence
	\[
	\int_0^\eps |v'|^2w_m\,dr=\infty.
	\]
	The regular solution has finite energy near $0$.
	Hence membership in $H^1((0,R),w_m\,dr)$ forces the coefficient of the nonregular homogeneous solution to vanish.
	For an inhomogeneous equation, it likewise excludes any nonregular homogeneous term added to a particular solution.
	Thus ``the regular branch at the origin'' means exactly the branch selected by finite energy; it is not an additional boundary condition.
	
	\subsubsection{Laguerre diagonalization}
	
	Set
	\[
	\nu=\frac{m+1}{2},\qquad s=\frac{r^2}{2},\qquad d\eta_\nu(s)=\frac{s^{\nu-1}e^{-s}}{\Gamma(\nu)}\,ds,
	\]
	and
	\[
	W_{m,\infty}:=\int_0^\infty w_m(r)\,dr =2^{\nu-1}\Gamma(\nu).
	\]
	After normalizing $w_m\,dr$ by $W_{m,\infty}$, define the unitary map
	\[
	(Uv)(s)=v(\sqrt{2s})
	\]
	onto $L^2(\eta_\nu)$.
	It transports the closed form according to
	\[
	\frac1{W_{m,\infty}}\int_0^\infty v'h'w_m\,dr
	=2\int_0^\infty s(Uv)'(Uh)'\,d\eta_\nu.
	\]
	Let
	\[
	A_\nu:=U(D^*D)U^{-1}.
	\]
	Then
	\[
	(D^*D+2)v=f
	\quad\Longleftrightarrow\quad
	(A_\nu+2)Uv=Uf.
	\]
	The operator $A_\nu$ is associated with the transformed form, and its interior expression is
	\begin{equation}\label{eq:radial-gamma-operator}
		A_\nu=-2\Bigl(s\frac{d^2}{ds^2} +(\nu-s)\frac d{ds}\Bigr).
	\end{equation}
	
	For $(a)_0=1$ and $(a)_k=a(a+1)\cdots(a+k-1)$, $k\ge1$, let
	\[
	L_k^\alpha(s) =\frac{s^{-\alpha}e^s}{k!} \frac{d^k}{ds^k}\bigl(e^{-s}s^{k+\alpha}\bigr), \qquad \alpha>-1.
	\]
	The following facts on $L_k^\alpha$ follow from \cite[\S\S18.3, 18.5, 18.6, 18.8, and 18.18(i)]{NISTDLMF} and some elementary calculations.
	
	\begin{lemma}
		\label{lem:standard-laguerre-facts}
		Let $m\ge1$ and $\nu=\frac{m+1}{2}$.
		Write $\ell_k=L_k^{\nu-1}$.
		The polynomials $\ell_k$ satisfy
		\[
		s\ell_k''+(\nu-s)\ell_k'+k\ell_k=0,
		\]
		form a complete orthogonal system in $L^2(\eta_\nu)$, and obey
		\[
		\ell_k(0)=\frac{(\nu)_k}{k!},
		\]
		and
		\begin{equation}\label{eq:laguerre-orthogonality-app}
			\int_0^\infty \ell_j\ell_k\,d\eta_\nu =\frac{(\nu)_k}{k!}\,\delta_{jk}.
		\end{equation}
		
		Consequently, every $F\in L^2(\eta_\nu)$ has the expansion
		\[
		F=\sum_{k=0}^\infty \widehat F_k\ell_k, \qquad
		\widehat F_k=\frac{k!}{(\nu)_k}\langle F,\ell_k\rangle_{L^2(\eta_\nu)},
		\]
		with convergence in $L^2(\eta_\nu)$, and
		\[
		\|F\|_{L^2(\eta_\nu)}^2 =\sum_{k=0}^\infty \frac{(\nu)_k}{k!}|\widehat F_k|^2.
		\]
		Moreover,
		\begin{equation}\label{eq:laguerre-half-moment-app}
			\int_0^\infty s^{-\frac12}\ell_k(s)\,d\eta_\nu(s)
			=\frac{\Gamma(\nu-\frac12)}{\Gamma(\nu)}\frac{(\frac12)_k}{k!}.
		\end{equation}
		The function $s^{-\frac12}$ belongs to $L^2(\eta_\nu)$ exactly when $\nu>1$.
	\end{lemma}
	
	\begin{lemma}[Laguerre representation of the resolvent]
		\label{lem:laguerre-resolvent}
		Let $F\in L^2(\eta_\nu)$ have the expansion
		\[
		F=\sum_{k=0}^\infty\widehat F_k\ell_k
		\]
		from Lemma~\ref{lem:standard-laguerre-facts}.
		Then the unique solution of
		\[
		(A_\nu+2)Z=F
		\]
		is
		\begin{equation}\label{eq:laguerre-resolvent-app}
			Z=(A_\nu+2)^{-1}F
			=\frac12\sum_{k=0}^\infty\frac{\widehat F_k}{k+1}\ell_k,
		\end{equation}
		with convergence in the graph norm of $A_\nu$.
		Equivalently, if $F=Uf$, then the unique solution of $(D^*D+2)v=f$ is $v=U^{-1}Z$.
	\end{lemma}
	
	\begin{proof}
		Since $\ell_k$ and $\ell_k'$ are bounded near $0$ and $\nu\ge1$, they are integrable there in the transformed form norm; the exponential factor in $\eta_\nu$ gives integrability at infinity.
		Thus $\ell_k$ belongs to the transformed form domain.
		Its differential equation gives, first for $H\in C_c^\infty(0,\infty)$,
		\[
		2\int_0^\infty s\ell_k'H'\,d\eta_\nu
		=2k\int_0^\infty\ell_kH\,d\eta_\nu.
		\]
		The form-core property extends this identity to every test function in the transformed form domain.
		The weak-domain characterization therefore gives
		\[
		\ell_k\in\operatorname{Dom}(A_\nu),
		\qquad A_\nu\ell_k=2k\ell_k.
		\]
		
		Thus the equation $(A_\nu+2)Z=F$, expanded in the complete system $(\ell_k)$, becomes
		\[
		2(k+1)z_k=\widehat F_k.
		\]
		This gives the candidate series in \eqref{eq:laguerre-resolvent-app}.
		If $Z_N$ denotes its $N$th partial sum, Parseval gives
		\begin{align*}
			\|Z-Z_N\|_{L^2(\eta_\nu)}^2
			&=\frac14\sum_{k>N}\frac{(\nu)_k}{k!}\frac{|\widehat F_k|^2}{(k+1)^2}\longrightarrow0,\\
			\|A_\nu(Z-Z_N)\|_{L^2(\eta_\nu)}^2
			&=\sum_{k>N}\frac{(\nu)_k}{k!}\Bigl(\frac{k}{k+1}\Bigr)^2|\widehat F_k|^2\longrightarrow0.
		\end{align*}
		Hence the series converges in the graph norm, and its limit solves $(A_\nu+2)Z=F$.
		Uniqueness follows from the nonnegativity of $A_\nu$.
	\end{proof}
	
	\vskip 20pt
	
	\section{\bf The local reduction}\label{sec:local-reduction}
	
	In this part, we derive a lower bound for $\alpha_\gamma(n)$ by refining the analytic route of Kolesnikov and Livshyts \cite{Livshyts}.
	
	For smooth data, the following estimate is obtained by Kolesnikov and Livshyts \cite[the proof of Lemma~2.3]{Livshyts}, based on the weighted Reilly formula of Kolesnikov and Milman \cite[Theorem~1.1]{milman2}.
	The proof below passes from a smooth source and classical solution to $g\in L^2(K,\gamma_n)$ and $u\in H^2(K)$.
	
	\begin{lemma}
		Let $K\in\mathcal K_{+,o}^n$, and let $f\in C^1(\partial K)$ and $g\in L^2(K,\gamma_n)$ satisfy \eqref{eq:compatibility}.
		If $u\in H^2(K)$ is a solution of \eqref{eq:neumann}, then
		\begin{equation}\label{eq:reilly-bound}
			\mathcal B_K(f) \le \int_K g^2\,d\gamma_n -\int_K\bigl(\|\Hess u\|_{\HS}^2+|\nabla u|^2\bigr)d\gamma_n.
		\end{equation}
	\end{lemma}
	
	\begin{proof}
		First suppose that $g$ is the restriction to $K$ of an ambient $C^\infty$ function.
		Lemma~\ref{lem:gaussian-neumann} gives $u\in\mathcal S_N(K)$.
		By \eqref{eq:reilly}, we have
		\[
		\begin{aligned}
			\int_Kg^2\,d\gamma_n ={}&\int_K\bigl(\|\Hess u\|_{\HS}^2+|\nabla u|^2\bigr)d\gamma_n\\
			&+\int_{\partial K}\bigl(H_\gamma f^2-2\langle\nabla_{\partial K}u,\nabla_{\partial K}f\rangle+\langle\II\nabla_{\partial K}u,\nabla_{\partial K}u\rangle\bigr)d\gamma_{\partial K}.
		\end{aligned}
		\]
		
		Since $K\in\mathcal K_{+,o}^n$, $\II$ is positive definite, and therefore
		\[
		\langle\II \nabla_{\partial K}u, \nabla_{\partial K}u \rangle-2\langle \nabla_{\partial K}u,\nabla_{\partial K}f\rangle \ge-\langle\II^{-1}\nabla_{\partial K}f,\nabla_{\partial K}f\rangle.
		\]
		Consequently,
		\[
		\int_Kg^2\,d\gamma_n \ge \int_K\bigl(\|\Hess u\|_{\HS}^2+|\nabla u|^2\bigr)d\gamma_n +\int_{\partial K} \Bigl( H_\gamma f^2 -\bigl\langle\II^{-1}\nabla_{\partial K}f, \nabla_{\partial K}f\bigr\rangle \Bigr)d\gamma_{\partial K},
		\]
		which is \eqref{eq:reilly-bound}.
		
		For a general $g\in L^2(K,\gamma_n)$, choose $\widetilde g_j\in C_c^\infty(\R^n)$ with $\widetilde g_j|_K\to g$ in $L^2(K,\gamma_n)$ and write 
		\[
		g_j=\widetilde g_j +\frac1{\gamma_n(K)} \int_K(g-\widetilde g_j)\,d\gamma_n\quad \text{on}\ K.
		\]
		Then
		\[
		g_j\to g\quad\text{in }L^2(K,\gamma_n), \quad\text{as}\ j\to \infty, \quad \text{and} \quad \int_K g_j\,d\gamma_n=\int_K g\,d\gamma_n.
		\]
		For each $j$, let $u_j$ solve \eqref{eq:neumann} with $g$ replaced by $g_j$, and normalize it by
		\[
		\int_K(u_j-u)\,d\gamma_n=0.
		\]
		Such a solution exists because $g_j$ satisfies the same compatibility condition as $g$.
		Since $g_j$ is smooth, Lemma~\ref{lem:gaussian-neumann} gives $u_j\in\mathcal S_N(K)$.
		Moreover, $u_j-u$ has Gaussian mean zero and satisfies
		\[
		L(u_j-u)=g_j-g\quad\text{in }K\quad \text{and} \quad (u_j-u)_{\nu_K}=0\quad\text{on }\partial K.
		\]
		By \eqref{eq:H2-estimate}, 
		\[
		\|u_j-u\|_{H^2(K)} \le C_K\|g_j-g\|_{L^2(K,\gamma_n)}\to 0,\quad\text{as}\ j\to \infty.
		\]
		The smooth case, applied to $u_j$, gives
		\[
		\mathcal B_K(f)\le \int_Kg_j^2\,d\gamma_n-\int_K\bigl(\|\Hess u_j\|_{\HS}^2+|\nabla u_j|^2\bigr)d\gamma_n.
		\]
		Letting $j\to \infty$ gives \eqref{eq:reilly-bound}.
	\end{proof}

	The next lemma combines the Kolesnikov--Livshyts local-to-global argument \cite[Lemma~3.1 and Corollary~1]{Livshyts} with their nonconstant source formulation \cite[Definition~6.2 and Lemma~6.3]{Livshyts}.
	We include the arguments to verify the slightly more flexible formulation needed here: the source may depend on the boundary datum, no symmetry is assumed, and only $C^1$ boundary data, an $L^2$ source, and an $H^2$ solution are required.
	
	\begin{lemma}
		\label{lem:nonconstant-criterion}
		Let $q>0$.
		Assume that for every $K\in\mathcal K_{+,o}^n$ and every $f\in C^1(\partial K)$ there exists $g\in L^2(K,\gamma_n)$ satisfying \eqref{eq:compatibility} such that an $H^2$ solution $u$ of \eqref{eq:neumann} satisfies
		\begin{equation}\label{eq:variance-energy}
			\int_K\bigl(\|\Hess u\|_{\HS}^2+|\nabla u|^2\bigr)d\gamma_n -\int_K(g-\bar g_K)^2\,d\gamma_n \ge q\,\gamma_n(K)\bar g_K^2,
		\end{equation}
		where
		\[
		\bar g_K=\frac1{\gamma_n(K)}\int_K g\,d\gamma_n.
		\]
		Then
		\[
		\gamma_n((1-t)K+tL)^q \ge(1-t)\gamma_n(K)^q+t\gamma_n(L)^q,
		\]
		for all $K,L\in\mathcal K_{+,o}^n$ and $t\in[0,1]$.
	\end{lemma}
	
	\begin{proof}
		Let $K,L\in\mathcal K_{+,o}^n$.
		By \eqref{eq:compatibility} and \eqref{eq:variance-energy},
		\[
		\int_K g^2\,d\gamma_n-\int_K\bigl(\|\Hess u\|_{\HS}^2+|\nabla u|^2\bigr)d\gamma_n\le(1-q)\gamma_n(K)\bar g_K^2=\frac{1-q}{\gamma_n(K)}\Bigl(\int_{\partial K}f\,d\gamma_{\partial K}\Bigr)^2.
		\]
		This, together with \eqref{eq:reilly-bound}, gives
		\begin{equation}\label{eq:local-q}
			\gamma_n(K)\mathcal B_K(f) \le(1-q) \Big(\int_{\partial K}f\,d\gamma_{\partial K}\Big)^2.
		\end{equation}
		
		Let $K_t=(1-t)K+tL$ and $F(t)=\gamma_n(K_t)$, $t\in [0,1]$. 
		Write $f_t=(h_L-h_K)\circ\nu_{K_t}$.
		Since $K,L\in\mathcal K_{+,o}^n$, we have  $h_K,h_L\in C^2(\Sph^{n-1})$ and $\nu_{K_t}$ is $C^1$, and therefore $f_t\in C^1(\partial K_t)$.
		Thus for $t\in(0,1)$, the variation formulas \eqref{eq:var} give
		\[
		F'(t)=\int_{\partial K_t}f_t\,d\gamma_{\partial K_t} \quad \text{and}\quad F''(t)=\mathcal B_{K_t}(f_t).
		\]
		Applying \eqref{eq:local-q} at $K_t$ yields that for all $t\in(0,1)$,
		\[
		F(t)F''(t)\le(1-q)F'(t)^2.
		\]
		Hence
		\[
		(F^q)''=qF^{q-2}\big(FF''+(q-1)F'^2\big)\le0.
		\]
		Thus $F^q$ is concave on $(0,1)$, and the continuity at the endpoints completes the proof.
	\end{proof}
	
	\vskip 20pt
	
	\section{\bf The upper bound}\label{sec:upper-bound}
	
	For the sake of completeness, we include the proof of the upper bound.
	The planar construction yielding the upper bound is due to Nayar and Tkocz \cite{Nayar}.
	Xiong and Yang \cite{XiongYang2026} extended the construction in higher dimensions.
	
	\begin{lemma}\label{lem:sharpness}
		\[
		\alpha_\gamma(n)\le q_n:=1-\frac{2}{n-1}\frac{\Gamma(\frac n2)^2}{\Gamma(\frac{n-1}{2})^2}.
		\]
	\end{lemma}
	
	\begin{proof}
		The normalizing constant of the Gaussian density cancels from the concavity inequality, so it is suitable to replace $\gamma_n$ as a measure $\mu$ on $\mathbb{R}^n$ with  
		\(
		d\mu(x)=e^{-\frac{|x|^2}{2}}\,dx.
		\)
		
		Fix $q>q_n$.
		We show that $q$ is not admissible in the definition of $\alpha_\gamma(n)$.
		
		\proofstep{1}{Construct convex sets.}

		Fix $0<\alpha<\frac\pi2$, let $e_n=(0,\ldots,0,1)$, and for $\eps>0$ set
		\begin{align*}
			A_\alpha&=\{(x',x_n)\in\R^{n-1}\times\R:x_n\ge |x'|\tan\alpha\},\\
			B_{\alpha,\eps}&=A_\alpha-\eps e_n=\{(x',x_n):x_n\ge |x'|\tan\alpha-\eps\}.
		\end{align*}
		Both sets are convex and contain the origin, and 
		\[
		\frac12A_\alpha+\frac12B_{\alpha,\eps}=A_\alpha-\frac\eps2e_n.
		\]
		Define $F_\alpha(t)=\mu(A_\alpha-te_n)$, $t\in \mathbb{R}$.
		Writing $r=|x'|$ gives
		\[
		F_\alpha(t)=\omega_{n-2}\int_0^\infty r^{n-2}e^{-\frac{r^2}{2}}\int_{r\tan\alpha-t}^\infty e^{-\frac{y^2}{2}}\,dy\,dr,
		\]
		where $\omega_k=\mathcal H^k(\Sph^k)$. In particular, $\omega_0=2$.
		For fixed $\alpha$, Gaussian domination permits two differentiations near $t=0$, and direct computations (see \cite[Step 3 in the proof of the upper bound]{XiongYang2026}) give that
		\begin{align*}
			F_\alpha'(0)&=\omega_{n-2}\int_0^\infty r^{n-2}e^{-\frac{r^2}{2\cos^2\alpha}}\,dr,\\
			F_\alpha''(0)&=\omega_{n-2}\tan\alpha\int_0^\infty r^{n-1}e^{-\frac{r^2}{2\cos^2\alpha}}\,dr.
		\end{align*}
		
		\proofstep{2}{Compute the narrow-cone limit.}

		Put $m=n-1$, $\beta=\frac\pi2-\alpha$, and
		\[
		c_0=\int_0^\infty s^{n-1}e^{-\frac{s^2}{2}}\,ds =2^{\frac n2-1}\Gamma\Bigl(\frac n2\Bigr), \qquad c_1=\int_0^\infty s^{n-2}e^{-\frac{s^2}{2}}\,ds =2^{\frac{n-3}{2}}\Gamma\Bigl(\frac{n-1}{2}\Bigr).
		\]
		Polar coordinates, followed by $r=s\cos\alpha$ in the two derivative formulas, give
		\begin{align*}
			F_\alpha(0)&=\omega_{n-2}c_0\int_0^\beta\sin^{m-1}\theta\,d\theta,\\
			F_\alpha'(0)&=\omega_{n-2}c_1\cos^m\alpha\quad\text{and}\quad 
			F_\alpha''(0)=\omega_{n-2}c_0\sin\alpha\cos^m\alpha.
		\end{align*}
		Moreover,
		\[
		q_n=1-\frac{c_0^2}{mc_1^2}.
		\]
		Since $\int_0^\beta\sin^{m-1}\theta\,d\theta\sim\frac{\beta^m}{m}$ and $\cos\alpha=\sin\beta\sim\beta$, we obtain
		\[
		\frac{F_\alpha'(0)}{F_\alpha(0)}\longrightarrow m\frac{c_1}{c_0},\qquad
		\frac{F_\alpha''(0)}{F_\alpha(0)}\longrightarrow m
		\]
		as $\alpha\uparrow\frac\pi2$.
		The chain rule gives
		\[
		(F_\alpha^q)''(0)=qF_\alpha(0)^q\Biggl[\frac{F_\alpha''(0)}{F_\alpha(0)}+(q-1)\Bigl(\frac{F_\alpha'(0)}{F_\alpha(0)}\Bigr)^2\Biggr].
		\]
		The expression in brackets converges to
		\[
		m^2\Bigl(\frac{c_1}{c_0}\Bigr)^2(q-q_n)>0,\quad \text{as}\ \alpha\uparrow\frac\pi2.
		\] 
		Choose $\alpha$ so that this expression is positive.
		Then $(F_\alpha^q)''(0)>0$, and Taylor's theorem gives, for all sufficiently small $\eps>0$,
		\[
		F_\alpha\Bigl(\frac\eps2\Bigr)^q-\frac12F_\alpha(0)^q-\frac12F_\alpha(\eps)^q=-\frac{\eps^2}{8}(F_\alpha^q)''(0)+o(\eps^2)<0.
		\]
		Thus the desired inequality fails strictly for $A_\alpha$ and $B_{\alpha,\eps}$.
		
		\proofstep{3}{Truncation to convex bodies.}

		For $R>0$, set
		\[
		A_R=A_\alpha\cap RB_2^n,\qquad B_R=B_{\alpha,\eps}\cap RB_2^n,\qquad C_R=\frac12A_R+\frac12B_R.
		\]
		The sets $A_R$ and $B_R$ belong to $\mathcal K_o^n$.
		Moreover, $A_R\uparrow A_\alpha$, $B_R\uparrow B_{\alpha,\eps}$ as $R\to \infty$, and $C_R$ is increasing in $R$.
		If $x=\frac12(a+b)$ with $a\in A_\alpha$ and $b\in B_{\alpha,\eps}$, then $x\in C_R$ whenever $R\ge\max\{|a|,|b|\}$.
		Consequently,
		\[
		C_R\uparrow\frac12A_\alpha+\frac12B_{\alpha,\eps}=A_\alpha-\frac\eps2e_n.
		\]
		Continuity from below gives convergence of all three Gaussian measures.
		Since the desired inequality fails strictly, it also fails for $A_R$ and $B_R$ when $R$ is sufficiently large.
		Thus $q$ is not admissible.
		Since $q>q_n$ was arbitrary, no exponent larger than $q_n$ is admissible, and hence $\alpha_\gamma(n)\le q_n$.
	\end{proof}
	
	\vskip 20pt
	
	\section{\bf The one-dimensional variational problem}
	\label{sec:one-dim}
	
	In this part, we solve a one-dimensional variational problem \eqref{eq:lambda-definition} for every real $m\ge1$. The geometric application in Section \ref{lift} uses $m=n-1$.
	
	Recall the radial weight from Subsection~\ref{ssec:radial-tools} and set
	\[
	W_{m,R}=\int_0^R w_m(r)\,dr,\qquad d\mu_{m,R}(r)=\frac{w_m(r)}{W_{m,R}}\,dr.
	\]
	Define
	\[
	\begin{aligned}
		\E_{m,R}:L^1((0,R),\mu_{m,R})&\to\mathbb R, & \E_{m,R}\phi&:=\int_0^R\phi\,d\mu_{m,R},\\
		\Var_{m,R}:L^2((0,R),\mu_{m,R})&\to[0,\infty), & \Var_{m,R}(\phi)&:=\E_{m,R}\big[(\phi-\E_{m,R}\phi)^2\big].
	\end{aligned}
	\]
	The weighted space $H^1((0,R),w_m\,dr)$ is understood as in Subsection~\ref{ssec:functional-tools}, with squared norm $\int_0^R(v^2+v'^2)w_m\,dr$.
	On this space, define
	\[
	T:H^1((0,R),w_m\,dr)\to L^2((0,R),w_m\,dr),\qquad (Tv)(r):=rv(r)-v'(r).
	\]
	The subscripts on $\E_{m,R}$ and $\Var_{m,R}$ will occasionally be omitted for convenience.
	For $g\in L^2((0,R),w_m\,dr)$ with $\E_{m,R}g=1$ and $v\in H^1((0,R),w_m\,dr)$, define
	\[
	\Phi_{m,R}(g,v):=\E_{m,R}\Big[v'^2+v^2+\frac{(g+Tv)^2}{m}\Big] -\Var_{m,R}(g),
	\]
	and
	\begin{equation}\label{eq:lambda-definition}
		\lambda_m(R)=\sup_{\substack{g\in L^2((0,R),w_m\,dr)\\ \E_{m,R}g=1}}\ \inf_{v\in H^1((0,R),w_m\,dr)}\Phi_{m,R}(g,v).
	\end{equation}
	
	An admissible pair \((g_*,v_*)\) for this max--min problem, if it exists, is called a saddle point if
	\[
	\Phi_{m,R}(g,v_*) \le \Phi_{m,R}(g_*,v_*) \le \Phi_{m,R}(g_*,v)
	\]
	for every admissible \((g,v)\). Here, a pair \((g,v)\) is admissible if $g\in L^2((0,R),w_m\,dr)$ with $\E_{m,R}g=1$ and $v\in H^1((0,R),w_m\,dr)$.
	
	The main result of this section is the following.
	
	\begin{theorem}\label{thm:radial-value}
		For every real $m\ge1$ and $R>0$,
		\[
		\lambda_m(R)\ge \kappa_m, \quad\text{where}\  \kappa_m=1-\frac2m \frac{\Gamma(\frac{m+1}{2})^2}{\Gamma(\frac m2)^2}.
		\]
		Moreover, the half-line limit satisfies
		\begin{equation}\label{eq:radial-limit}
			\lim_{R\to\infty}\lambda_m(R)=\kappa_m.
		\end{equation}
		The supremum in \eqref{eq:lambda-definition} is attained for every $R>0$.
	\end{theorem}
	
	For $m=n-1$, the constant in Theorem~\ref{thm:radial-value} satisfies $\kappa_{n-1}=q_n$.
	
	For $m>1$, we eliminate $g$ by a saddle reduction, show that $\lambda_m(R)$ is nonincreasing in $R$, and evaluate the worst case on the half-line using the Laguerre facts from Subsection~\ref{ssec:radial-tools}.
	The case $m=1$ is handled by a degenerate saddle argument separately.
	Explicit maximizing profiles, which play no role in the sharp radial estimate for $\lambda_m$, are given in Appendix~\ref{app:explicit-profiles} for completeness.
	
	\subsection{Saddle reduction away from \texorpdfstring{$m=1$}{m=1}}
	
	We first treat the case $m>1$.
	
	Let $D=\frac{d}{dr}$. So $T=r-D$.
	To derive the Euler equation and its endpoint condition, we use weighted integration by parts on $(0,R)$.
	Throughout this subsection, $D^*$ and $T^*$ denote only the formal adjoint differential expressions with respect to $w_m\,dr$; they do not specify an operator domain or a boundary condition at $R$.
	For $\phi,\psi\in C_c^\infty(0,R)$, the formal adjoints are characterized by
	\[
	\int_0^R (D\phi)\psi w_m\,dr =\int_0^R \phi(D^*\psi)w_m\,dr, \qquad \int_0^R (T\phi)\psi w_m\,dr =\int_0^R \phi(T^*\psi)w_m\,dr.
	\]
	Since $\frac{w_m'}{w_m}=\frac{m}{r}-r$, direct integration by parts gives
	\begin{equation}\label{eq:adjoints}
		\begin{aligned}
			D^*&=-D+r-\frac mr, &T^*&=D+\frac mr,\\
			D^*D&=-D^2+\big(r-\frac mr\big)D, &T^*T&=D^*D+(m+1).
		\end{aligned}
	\end{equation}
	If $\phi$ and $\psi$ are smooth, vanish near $0$, and have traces at $R$, the same calculation retains the endpoint terms
	\begin{align*}
		\int_0^R (D\phi)\psi w_m\,dr&=\int_0^R\phi(D^*\psi)w_m\,dr+\phi(R)\psi(R)w_m(R),\\
		\int_0^R (T\phi)\psi w_m\,dr&=\int_0^R\phi(T^*\psi)w_m\,dr-\phi(R)\psi(R)w_m(R).
	\end{align*}
	The argument in Subsection~\ref{ssec:radial-tools} shows that functions of this type are dense in $H^1((0,R),w_m\,dr)$.
	Thus the weak integrations by parts used below have no contribution from the singular endpoint $0$, while the trace term at $R$ remains.
	
	In the variational argument below, setting the coefficient of the arbitrary trace at $R$ equal to zero gives the natural boundary condition for the minimizer.
	
	\begin{lemma}\label{lem:saddle-reduction}
		Let $m>1$.
		Then
		\begin{equation}\label{eq:dual-functional}
			\lambda_m(R) =\inf_{v\in H^1((0,R),w_m\,dr)}J_{m,R}(v),
		\end{equation}
		where $J_{m,R}:H^1((0,R),w_m\,dr)\to\mathbb R$ is defined by
		\begin{equation}\label{eq:J-functional}
			J_{m,R}(v):=\E_{m,R}(v'^2+v^2) +\frac{\Var_{m,R}(Tv)}{m-1} +\frac{(1+\E_{m,R}Tv)^2}{m}.
		\end{equation}
		The minimizer is unique and denoted by $v_{m,R}$.
		Let $z=z_{m,R}\in H^1((0,R),w_m\,dr)$ be the unique solution of
		\begin{align}
			(D^*D+2)z=\frac1r,& \quad \text{in }(0,R),\label{eq:z-ode}\\
			mz'(R)=Rz(R)-1,&\label{eq:z-boundary}
		\end{align}
		on the regular branch at the origin, and  $H_m(R)=\E_{m,R}(Tz_{m,R})$.
		Then
		\begin{equation}\label{eq:c-lambda-formula}
			v_{m,R}=-c_{m,R}z_{m,R},\qquad c_{m,R}=\frac{m-1}{m-H_m(R)},\qquad \lambda_m(R)=1-c_{m,R} =\frac{1-H_m(R)}{m-H_m(R)},
		\end{equation}
		and the maximizing right-hand side in \eqref{eq:lambda-definition} is
		\begin{equation}\label{eq:g-optimal}
			g_{m,R}(r)=\frac{m-Tz_{m,R}(r)}{m-H_m(R)} =\frac{m-rz_{m,R}(r)+z_{m,R}'(r)}{m-H_m(R)}.
		\end{equation}
	\end{lemma}
	
	\begin{proof}
		We suppress $m,R$ from $\Phi_{m,R}$, $J_{m,R}$, $\E_{m,R}$, and $\Var_{m,R}$ in this proof for convenience.
		
		\proofstep{1}{Elimination of $g$.}
		
		\noindent\hspace*{2em}Let $v\in H^1((0,R),w_m\,dr)$ and  $\tau_v=\E(Tv)$.
		Then for every $g\in L^2((0,R),w_m\,dr)$ with $\E g=1$, 
		\[
		\frac1m\E(g+Tv)^2-\Var(g)=\frac{(1+\tau_v)^2}{m}+\frac{\Var(Tv)}{m-1}-\frac{m-1}{m}\E\Big[\Big(g-1-\frac{Tv-\tau_v}{m-1}\Big)^2\Big].
		\]
		Since $m>1$, the unique maximizer over all such $g$ is
		\begin{equation}\label{eq:g-fixed-v}
			g=1+\frac{Tv-\E(Tv)}{m-1},
		\end{equation}
		and consequently
		\[
		\sup_{\E g=1}\Phi(g,v)=J(v).
		\]
		
		On $H^1((0,R),w_m\,dr)$, define the symmetric bilinear form
		\[
		Q(u,h):=\E(u'h'+uh)+\frac{\E\big[(Tu-\E Tu)(Th-\E Th)\big]}{m-1}+\frac{\E(Tu)\E(Th)}{m}.
		\]
		Then $J(v)=Q(v,v)+2\frac{\E(Tv)}{m}+\frac{1}{m}$, and
		\[
		Q(h,h)=\E(h'^2+h^2)+\frac{\Var(Th)}{m-1}+\frac{(\E Th)^2}{m}\ge \E(h'^2+h^2).
		\]
		Since $T:H^1((0,R),w_m\,dr)\to L^2(\mu_{m,R})$ is continuous, $Q$ is bounded and coercive and $h\mapsto\E(Th)$ is bounded.
		For $v,h\in H^1((0,R),w_m\,dr)$, direct differentiation gives
		\[
		\frac12\frac{d}{d\varepsilon}J(v+\varepsilon h)\Big|_{\varepsilon=0}=Q(v,h)+\frac1m\E(Th).
		\]
		Hence $v_*$ is a critical point of $J$ exactly when
		\[
		Q(v_*,h)=-\frac1m\E(Th)\qquad\text{for every }h\in H^1((0,R),w_m\,dr),
		\]
		and Lax--Milgram gives the unique $v_*$.
		Moreover, for $0<t<1$ and $v, w\in H^1((0,R),w_m\,dr)$ with $v\ne w$,
		\[
		tJ(v)+(1-t)J(w)-J(tv+(1-t)w)=t(1-t)Q(v-w,v-w)>0,
		\]
		where the last inequality follows from the displayed coercive bound.
		Thus $J$ is strictly convex, and its critical point $v_*$ is the unique minimizer. Write $v_*=v_{m,R}$.
		
		\proofstep{2}{Euler equation and natural boundary condition.}
		
		\noindent\hspace*{2em}Let $\tau=\E(Tv_*)$.
		Since $v_*$ minimizes $J$, for every $h\in H^1((0,R),w_m\,dr)$,
		\[
		0=\frac12\frac{d}{d\varepsilon}J(v_*+\varepsilon h)\Big|_{\varepsilon=0}
		=\E(v_*'h'+v_*h)+\frac1{m-1}\E\big[(Tv_*-\tau)Th\big]+\frac{1+\tau}{m}\E(Th).
		\]
		
		First take $h\in C_c^\infty(0,R)$ to determine the interior equation.
		Since $T^*T=D^*D+(m+1)$ and $T^*1=\frac{m}{r}$, the above first-variation identity gives
		\[
		\begin{aligned}
			0&=D^*Dv_*+v_*+\frac{T^*(Tv_*-\tau)}{m-1}+\frac{1+\tau}{m}T^*1\\
			&=\frac{m}{m-1}(D^*D+2)v_*+\frac{m-1-\tau}{(m-1)r}
		\end{aligned}
		\]
		Thus
		\[
		(D^*D+2)v_*=-\frac{m-1-\tau}{m}\frac1r.
		\]
		Since $m>1$, the right-hand side belongs to $L^2((0,R),\mu_{m,R})$.
		On $(\frac{R}{2},R)$, the weight is bounded above and below by positive constants, and
		\[
		v_*''=\big(r-\frac mr\big)v_*'+2v_*+\frac{m-1-\tau}{mr}
		\]
		has right-hand side in $L^2$.
		Thus $v_*\in H^2((\frac{R}{2},R))$, so $v_*'(R)$ and $Tv_*(R)$ are well defined.
		
		To obtain the boundary condition, take $h\in C^\infty([0,R])$ vanishing near $0$, with arbitrary $h(R)$.
		Applying the Green identities in the first-variation formula gives
		\[
		0=\frac{w_m(R)}{W_{m,R}}h(R)\Big[v_*'(R)-\frac{Tv_*(R)-\tau}{m-1}-\frac{1+\tau}{m}\Big].
		\]
		Since $\frac{w_m(R)}{W_{m,R}}>0$ and $h(R)$ is arbitrary, 
		\[
		0=v_*'(R)-\frac{Tv_*(R)-\tau}{m-1}-\frac{1+\tau}{m}.
		\]
		Then $Tv_*(R)=Rv_*(R)-v_*'(R)$ gives
		\begin{equation}\label{eq:v-boundary}
			mv_*'(R)=Rv_*(R)+\frac{m-1-\tau}{m}.
		\end{equation}
		
		\proofstep{3}{Identification of the minimizer and representation of its value.}
		
		\noindent\hspace*{2em}Taking $h=v_*$ in the first-variation identity gives
		\[
		\E(v_*'^2+v_*^2)+\frac{\Var(Tv_*)}{m-1}+\frac{(1+\tau)\tau}{m}=0.
		\]
		By the definition of $J$, 
		\begin{equation}\label{eq:J-scaling}
			J(v_*)=\frac{1+\tau}{m}.
		\end{equation}
		Define
		\[
		c=c_{m,R}:=\frac{m-1-\tau}{m}=1-J(v_*).
		\]
		Since $0\le J(v_*)\le J(0)=\frac{1}{m}$, we have $c\ge1-\frac{1}{m}>0$, and the equation and boundary condition from Step~2 become
		\[
		(D^*D+2)v_*=-\frac c r,\qquad mv_*'(R)=Rv_*(R)+c.
		\]
		Therefore $z=-\frac{v_*}{c}$ belongs to the form domain and solves \eqref{eq:z-ode}--\eqref{eq:z-boundary}.
		The other local homogeneous branch behaves as $r^{1-m}$, whose derivative is not square-integrable against $w_m\,dr$ for $m>1$, so $z$ is on the regular branch at $0$.
		
		To prove the uniqueness of $z$, let $\psi$ be the difference of two regular solutions.
		Then
		\[
		(D^*D+2)\psi=0,\qquad m\psi'(R)=R\psi(R).
		\]
		Regularity at $0$ makes the boundary term there vanish.
		Multiplication by $\psi$, integration by parts, and completion of the square give
		\[
		\begin{aligned}
			0&=\int_0^R(\psi'^2+2\psi^2)w_m\,dr-\frac Rm\psi(R)^2w_m(R)\\
			&=\int_0^R\Bigl(\psi'-\frac rm\psi\Bigr)^2w_m\,dr+\frac{m-1}{m}\int_0^R\Bigl(1+\frac{r^2}{m}\Bigr)\psi^2w_m\,dr.
		\end{aligned}
		\]
		Since $m>1$, the sum can vanish only when $\psi=0$.
		
		Since $v_*=-cz$, we have
		\[
		\tau=-cH_m(R) \quad \text{and}\quad c=\frac{m-1-\tau}{m}.
		\]
		Eliminating $\tau$ gives
		\[
		c\bigl(m-H_m(R)\bigr)=m-1\quad \text{and}\quad c=\frac{m-1}{m-H_m(R)}.
		\]
		Hence
		\[
		J(v_*)=1-c=\frac{1-H_m(R)}{m-H_m(R)}.
		\]
		So the maximizer associated with $v_*$ in Step~1 is
		\[
		g_*=g_{m,R}:=1+\frac{-cTz+cH_m(R)}{m-1}=\frac{m-Tz}{m-H_m(R)},
		\]
		which is \eqref{eq:g-optimal}.
		It is admissible because $z\in H^1$, $m-H_m(R)=\frac{m-1}{c}>0$, and its construction in \eqref{eq:g-fixed-v} gives $\E g_*=1$.
		
		\proofstep{4}{Recovery of the max--min value.}
		
		\noindent\hspace*{2em}Step~1 computes $\inf_v\sup_g\Phi(g,v)$, so it remains to show that $(g_*,v_*)$ is exactly a saddle point.
		
		For $v,h\in H^1((0,R),w_m\,dr)$, write  $v_\varepsilon=v+\varepsilon h$.
		Since $T$ is linear,
		\[
		\begin{aligned}
			\Phi(g_*,v_\varepsilon)
			={}&\Phi(g_*,v)+2\varepsilon\E\Big(v'h'+vh+\frac{(g_*+Tv)Th}{m}\Big)\\
			&+\varepsilon^2\E\Big(h'^2+h^2+\frac{(Th)^2}{m}\Big).
		\end{aligned}
		\]
		Write $\tau=\E(Tv_*)$. Formula \eqref{eq:g-fixed-v} gives
		\[
		\frac{g_*+Tv_*}{m}=\frac{Tv_*-\tau}{m-1}+\frac{1+\tau}{m}.
		\]
		Thus the coefficient of $\varepsilon$ vanishes at $v=v_*$ by the first-variation identity in Step~2.
		For $h\ne0$, the coefficient of $\varepsilon^2$ is positive because $\E(h'^2+h^2)>0$.
		Thus $v_*$ is the unique minimizer of $v\mapsto\Phi(g_*,v)$. Moreover, 
		Step~1 also shows that $g_*$ is the unique maximizer of $g\mapsto\Phi(g,v_*)$ under $\E g=1$. 
		Therefore, for every admissible $g$ and $v$,
		\[
		\Phi(g,v_*)\le\Phi(g_*,v_*)\le\Phi(g_*,v).
		\]
		
		For every admissible $g$, the left saddle inequality gives
		\[
		\inf_v\Phi(g,v)\le\Phi(g,v_*)\le\Phi(g_*,v_*),
		\]
		and hence $\lambda_m(R)\le\Phi(g_*,v_*)$.
		Conversely, $g_*$ is admissible and the right saddle inequality gives
		\[
		\lambda_m(R)\ge\inf_v\Phi(g_*,v)=\Phi(g_*,v_*).
		\]
		Thus
		\[
		\lambda_m(R)=\Phi(g_*,v_*)=J(v_*)=\inf_v J(v).
		\]
		This, together with Step~3, proves \eqref{eq:dual-functional} and \eqref{eq:c-lambda-formula}.
	\end{proof}
	
	\subsection{Monotonicity in the segment length}
	
	In this part, we prove that $\lambda_m(R)$ is nonincreasing in $R$.
	
	\begin{lemma}\label{lem:boundary-estimates}
		Let $m>1$ and let $z_{m,R}$ solve \eqref{eq:z-ode}--\eqref{eq:z-boundary}.
		Then
		\[
		z_{m,R}(R)>0,\qquad z_{m,R}'(R)<0,\qquad m z_{m,R}'(R)^2+z_{m,R}(R)^2\le -z_{m,R}'(R),
		\]
		and
		\[
		1-H_m(R)\ge-(m-1)z_{m,R}'(R).
		\]
	\end{lemma}
	
	\begin{proof}
		Write $z=z_{m,R}$ and define
		\[
		\Xi=z'+\frac{1-rz}{m}.
		\]
		
		\proofstep{1}{Endpoint values of $\Xi$.}
		
		\noindent\hspace*{2em}We first prove that $\Xi$ extends continuously to $[0,R]$ and vanishes at both endpoints.
		
		Since $z$ is on the regular branch, it is bounded near $0$, and \eqref{eq:z-ode} becomes $(w_mz')'=2w_mz-\frac{w_m}{r}$.
		The right-hand side is integrable at $0$, so $w_mz'$ has a finite limit there.
		Finite energy forces this limit to be zero. Otherwise $|z'|^2w_m$ would be comparable to $w_m^{-1}\sim r^{-m}$.
		Hence
		\[
		w_m(r)z'(r)=2\int_0^r w_m(s)z(s)\,ds-\int_0^r s^{m-1}e^{-\frac{s^2}{2}}\,ds=-\frac{r^m}{m}+O(r^{m+1}).
		\]
		Since $w_m(r)=r^m(1+O(r^2))$, we have $z'(r)=-\frac{1}{m}+O(r)$.
		Also $rz(r)=O(r)$, and therefore $\Xi(r)=O(r)$.
		Thus $\Xi$ extends continuously to $0$ with $\Xi(0)=0$.
		
		At $R$, the boundary condition gives $z'(R)=\frac{Rz(R)-1}{m}$, so
		\[
		\Xi(R)=\frac{Rz(R)-1}{m}+\frac{1-Rz(R)}m=0.
		\]
		
		\proofstep{2}{Positivity of $\Xi$ and monotonicity of $Tz$.}
		
		\noindent\hspace*{2em}Using \eqref{eq:z-ode}, direct differentiation gives
		\[
		(Tz)'=\frac mr\Xi.
		\]
		The same equation also gives
		\[
		-\Xi''+\Big(r-\frac mr+\frac{2r}{r^2+m}\Big)\Xi'+\Big(\frac m{r^2}+\frac{r^2+5m-2}{r^2+m}\Big)\Xi=\frac{2(m-1)}{m(r^2+m)}.
		\]
		The zeroth-order coefficient and the right-hand side are positive because $m>1$.
		Together with the endpoint values from Step~1, the maximum principle gives $\Xi>0$ on $(0,R)$.
		Consequently, $(Tz)'>0$ on $(0,R)$.
		
		\proofstep{3}{Boundary estimates for $z$.}
		
		\noindent\hspace*{2em}By Step~2, $Tz(R)>Tz(0)=\frac{1}{m}$.
		The boundary condition gives
		\[
		Tz(R)=\frac1m+\frac{m-1}{m}Rz(R),
		\]
		and hence $z(R)>0$.
		The Hopf lemma at the regular endpoint $R$ gives $\Xi'(R)<0$.
		Differentiating $\Xi$ and using \eqref{eq:z-ode} and \eqref{eq:z-boundary}, we obtain
		\[
		\Xi'(R)=\frac{m-1}{m^2} \bigl((m+R^2)z(R)-R\bigr),
		\]
		and therefore
		\[
		0<z(R)<\frac{R}{R^2+m}.
		\]
		In particular, $Rz(R)<1$, so the boundary condition gives $z'(R)<0$.
		Moreover,
		\[
		-z'(R)-m z'(R)^2-z(R)^2 =-\frac{z(R)}m\bigl((m+R^2)z(R)-R\bigr)\ge0.
		\]
		This is the required quadratic boundary estimate.
		
		\proofstep{4}{The estimate for $H_m(R)$.}
		
		\noindent\hspace*{2em}By Step~2, $Tz$ is increasing, so $H_m(R)=\E_{m,R}(Tz)\le Tz(R)$.
		The boundary condition gives $Tz(R)=1+(m-1)z'(R)$, and therefore
		\[
		1-H_m(R)\ge1-Tz(R)=-(m-1)z'(R).
		\]
		This completes the proof.
	\end{proof}
	
	$\lambda_m'$ is computed in the following.
	
	\begin{lemma}\label{lem:envelope}
		For $m>1$, the value $\lambda_m(R)$ is differentiable and
		\[
		\lambda_m'(R)=\frac{w_m(R)}{W_{m,R}} [mv_{m,R}'(R)^2+v_{m,R}(R)^2 -\lambda_m(R)(1-\lambda_m(R))].
		\]
	\end{lemma}

	\begin{proof}
		We divide the proof into three steps.
		
		\proofstep{1}{Prove the smooth dependence of the minimizer.}
		
		\noindent\hspace*{2em}Let
		\[
		X=H^1((0,1),s^m\,ds)\quad\text{and}\quad V_R(s)=v_{m,R}(Rs).
		\]
		Denote by $\widehat J_R$ the pullback of $J_{m,R}$ to $X$.
		Write 
		\[
		q_R(s)=\frac{R^{m+1}}{W_{m,R}}e^{-\frac{R^2s^2}{2}}\quad\text{and}\quad  \widehat{\E}_R F=\int_0^1 F(s)q_R(s)s^m\,ds.
		\]
		Under $r=Rs$, we have 
		\[
		d\mu_{m,R}(Rs)=q_R(s)s^m\,ds, \qquad v'(Rs)=R^{-1}V'(s) \quad\text{and}\quad Tv(Rs)=RsV(s)-R^{-1}V'(s).
		\]
		Therefore,
		\[
		\begin{aligned}
			\widehat J_R(V)={}&\widehat{\E}_R\bigl(R^{-2}V'^2+V^2\bigr)+\frac{\widehat{\E}_R\bigl[\bigl(RsV-R^{-1}V'\bigr)^2\bigr]-\bigl[\widehat{\E}_R\bigl(RsV-R^{-1}V'\bigr)\bigr]^2}{m-1}\\
			&+\frac{\bigl[1+\widehat{\E}_R\bigl(RsV-R^{-1}V'\bigr)\bigr]^2}{m}.
		\end{aligned}
		\]
		
		Thus the $R$-dependent coefficients are $q_R(s)$, $R^{-2}$, $Rs$, and $R^{-1}$, together with their products.
		Since $W_{m,R}'=w_m(R)$,
		\[
		\partial_R q_R(s)=q_R(s) \Big(\frac{m+1}{R}-\frac{w_m(R)}{W_{m,R}}-Rs^2\Big).
		\]
		Hence, on every compact interval of radii, $q_R$ is bounded above and below by positive constants, and both $q_R$ and $\partial_Rq_R$ are uniformly bounded on $[0,1]$.
		The remaining coefficients are smooth in $R$.
		Consequently, the Euler map
		\[
		\mathcal E:(0,\infty)\times X\longrightarrow X^*,\qquad
		\mathcal E(R,V)[H]=\frac d{dt}\Big|_{t=0}\widehat J_R(V+tH),\quad H\in X,
		\]
		is $C^1$.
		Moreover, the quadratic part of $\widehat J_R$ is uniformly coercive on $X$ since
		\[
		\widehat{\E}_R\bigl(R^{-2}H'^2+H^2\bigr) \ge c\|H\|_X^2
		\]
		for every $H\in X$, with $c>0$ locally uniform in $R$.
		
		For fixed $R$, the equation $\mathcal E(R,V)=0$ is an affine variational equation whose bilinear part is bounded and coercive.
		The Lax--Milgram theorem gives its unique solution $V_R$.
		The coercivity of the quadratic part makes $\widehat J_R$ strictly convex, so $V_R$ is also its unique minimizer.
		
		Fix $R_0>0$.
		We have $\mathcal E(R_0,V_{R_0})=0$, and $\mathcal E$ is $C^1$.
		Since $V\mapsto\mathcal E(R_0,V)$ is affine, its derivative $\partial_V\mathcal E(R_0,V_{R_0}):X\to X^*$ is the operator associated with twice the quadratic part of $\widehat J_{R_0}$.
		The coercivity estimate above and the Lax--Milgram theorem show that this derivative is a bounded isomorphism.
		
		The Banach-space implicit function theorem \cite[Theorem~15.1 and Corollary~15.1, pp.~148--150]{Deimling1985} therefore gives an interval $I$ with $R_0\in I$ and a $C^1$ map $\widetilde V:I\to X$ such that $\widetilde V(R_0)=V_{R_0}$ and $\mathcal E(R,\widetilde V(R))=0$ for every $R\in I$.
		For each $R\in I$, the uniqueness established above gives $\widetilde V(R)=V_R$.
		Since $R_0$ was arbitrary, $R\mapsto V_R$ is $C^1$ on $(0,\infty)$.
		
		Since $\widehat J_R$ is $C^1$ in $(R,V)$ and $\lambda_m(R)=\widehat J_R(V_R)$, $\lambda_m$ is differentiable.
		
		\proofstep{2}{Differentiation of $\lambda_m(R)$.}
		
		\noindent\hspace*{2em}Fix $R_0>0$ and set $v_0=v_{m,R_0}$.
		The endpoint regularity proved in Lemma~\ref{lem:saddle-reduction} gives well-defined traces $v_0(R_0)$ and $v_0'(R_0)$.
		Choose $0<\delta<\frac{R_0}{2}$ and extend $v_0$ to $(0,R_0+\delta)$ by
		\[
		\widetilde v(r)=
		\begin{cases}
			v_0(r),&r\le R_0,\\
			v_0(R_0)+v_0'(R_0)(r-R_0),&r>R_0.
		\end{cases}
		\]
		For $|R-R_0|<\delta$, let $\Psi(R)=J_{m,R}(\widetilde v|_{(0,R)})$.
		By \eqref{eq:dual-functional},
		\[
		\lambda_m(R)\le\Psi(R),
		\qquad \lambda_m(R_0)=\Psi(R_0).
		\]
		The matching extension makes $\widetilde v'^2+\widetilde v^2$, $T\widetilde v$, and $(T\widetilde v)^2$ integrable with respect to $w_m\,dr$ and continuous at $R_0$.
		Hence $\Psi$ is differentiable at $R_0$, since for every such fixed integrand $F$,
		\[
		\left.\frac d{dR}\E_{m,R}F\right|_{R=R_0}
		=\frac{w_m(R_0)}{W_{m,R_0}}\bigl(F(R_0)-\E_{m,R_0}F\bigr).
		\]
		Step~1 gives the differentiability of $\lambda_m$; the touching relation above therefore yields $\lambda_m'(R_0)=\Psi'(R_0)$.
		Applying the last identity to the three integrands and then writing $R=R_0$, $v=v_0=v_{m,R}$, $\tau=\E_{m,R}(Tv)$, and suppressing the subscripts on $\E_{m,R}$ and $\Var_{m,R}$, we obtain
		\begin{align*}
			\frac{W_{m,R}}{w_m(R)}\lambda_m'(R)={}&v'(R)^2+v(R)^2-\E(v'^2+v^2)\\
			&+\frac{(Tv(R)-\tau)^2-\Var(Tv)}{m-1}
			+\frac{2(1+\tau)(Tv(R)-\tau)}m.
		\end{align*}
		
		\proofstep{3}{Reduction to boundary data.}
		
		\noindent\hspace*{2em}Write $\lambda=\lambda_m(R)$.
		The boundary condition \eqref{eq:v-boundary} and \eqref{eq:J-scaling} give
		\[
		1+\tau=m\lambda \quad \text{and} \quad Tv(R)-\tau=(m-1)\bigl(v'(R)-\lambda\bigr).
		\]
		Moreover, $J(v)=\lambda$ and $1+\tau=m\lambda$ imply
		\[
		\E(v'^2+v^2)+\frac{\Var(Tv)}{m-1} =\lambda-m\lambda^2.
		\]
		Substitution into the identity from Step~2 gives
		\begin{align*}
			\frac{W_{m,R}}{w_m(R)}\lambda_m'(R)={}&v'(R)^2+v(R)^2-\lambda+m\lambda^2+(m-1)\bigl(v'(R)-\lambda\bigr)^2+2(m-1)\lambda\bigl(v'(R)-\lambda\bigr)\\
			={}&mv'(R)^2+v(R)^2-\lambda(1-\lambda),
		\end{align*}
		which is the claimed formula.
	\end{proof}
	
	Combining these results above, we obtain the main result in this part.
	
	\begin{lemma}
		\label{lem:finite-monotonicity}
		If $m>1$, then $R\mapsto\lambda_m(R)$ is nonincreasing.
	\end{lemma}
	
	\begin{proof}
		By Lemma~\ref{lem:boundary-estimates}, \eqref{eq:c-lambda-formula} and $m-H_m(R)=(m-1)+(1-H_m(R))$, we have
		\[
		1-H_m(R)>0,\qquad0<\lambda_m(R)<1 \quad\text{and}\quad -z_{m,R}'(R) \le\frac{1-H_m(R)}{m-1} =\frac{\lambda_m(R)}{1-\lambda_m(R)}.
		\]
		This, together with $v_{m,R}=-(1-\lambda_m(R))z_{m,R}$ and Lemma~\ref{lem:boundary-estimates} again, implies that
		\[
		mv_{m,R}'(R)^2+v_{m,R}(R)^2 \le-(1-\lambda_m(R))^2z_{m,R}'(R) \le\lambda_m(R)(1-\lambda_m(R)),
		\]
		By Lemma~\ref{lem:envelope}, we have $\lambda_m'(R)\le0$.
	\end{proof}
	
	\subsection{The half-line value}
	
	This subsection computes $\lim_{R\to\infty}\lambda_m(R)$, which by Lemma~\ref{lem:finite-monotonicity} equals the infimum of $\lambda_m(R)$ over all $R>0$.
	
	\begin{lemma}\label{lem:half-line-value}
		For $m>1$,
		\[
		\lim_{R\to\infty}\lambda_m(R) =1-\frac2m \frac{\Gamma(\frac{m+1}{2})^2}{\Gamma(\frac m2)^2}.
		\]
	\end{lemma}
	
	\begin{proof}
		We divide the proof into four steps.
		
		\proofstep{1}{The half-line problem.}
		
		\noindent\hspace*{2em}Define the probability measure
		\[
		d\mu_{m,\infty}(r)=\frac{w_m(r)}{W_{m,\infty}}\,dr.
		\]
		Because $m>1$, the function $\frac{1}{r}$ belongs to $L^2(\mu_{m,\infty})$.
		The bilinear form
		\[
		(z,h)\longmapsto \int_0^\infty(z'h'+2zh)\,d\mu_{m,\infty}
		\]
		is continuous and coercive on $H^1((0,\infty),\mu_{m,\infty})$.
		The Lax--Milgram theorem therefore gives a unique $z_{m,\infty}$ in this space such that
		\[
		\int_0^\infty(z_{m,\infty}'h'+2z_{m,\infty}h)\,d\mu_{m,\infty} =\int_0^\infty\frac h r\,d\mu_{m,\infty}
		\]
		for every $h\in H^1((0,\infty),\mu_{m,\infty})$.
		In the notation of Subsection~\ref{ssec:radial-tools}, this is precisely the finite-energy solution of $(D^*D+2)z=\frac1r$; finite energy selects the regular branch at the origin.
		
		\proofstep{2}{Pass from finite intervals to the half-line.}
		
		\noindent\hspace*{2em}Lemma~\ref{lem:boundary-estimates} and \eqref{eq:z-boundary} imply
		\[
		0<-z_{m,R}'(R)\le\frac1m, \qquad 0<z_{m,R}(R)\le\frac1{\sqrt m}.
		\]
		Testing the equation for $z_{m,R}$ with $z_{m,R}$ and using the regular behavior at $0$ gives
		\[
		\E_{m,R}(z_{m,R}'{}^2+2z_{m,R}^2) -\frac{w_m(R)}{W_{m,R}}z_{m,R}'(R)z_{m,R}(R) =\E_{m,R}\frac{z_{m,R}}r.
		\]
		The $L^2(\mu_{m,R})$ norms of $\frac{1}{r}$ are uniformly bounded for $R\ge1$.
		Cauchy--Schwarz therefore gives
		\[
		\sup_{R\ge1}\E_{m,R} (z_{m,R}'{}^2+z_{m,R}^2)<\infty.
		\]
		
		Extend each $z_{m,R}$ to $(0,\infty)$ by the constant value $z_{m,R}(R)$, and use the same symbol for this extension.
		It belongs to $H^1((0,\infty),\mu_{m,\infty})$, and
		\[
		\int_0^\infty(z_{m,R}'{}^2+z_{m,R}^2)\,d\mu_{m,\infty}=\frac{W_{m,R}}{W_{m,\infty}}\E_{m,R}(z_{m,R}'{}^2+z_{m,R}^2)+z_{m,R}(R)^2\mu_{m,\infty}((R,\infty)),
		\]
		so these extensions are bounded in the half-line Sobolev space.
		Given any sequence $R_j\to\infty$, take a weakly convergent subsequence.
		For every $h\in C_c^\infty(0,\infty)$, the support of $h$ lies in $(0,R_j)$ for all sufficiently large $j$, and the interior weak equation passes to the limit.
		The statement in Subsection~\ref{ssec:radial-tools} extends the resulting identity to every half-line Sobolev test function.
		The limit must therefore be the unique solution $z_{m,\infty}$ from Step~1.
		Since this argument applies to every sequence, the whole family converges weakly, that is
		\[
		z_{m,R}\rightharpoonup z_{m,\infty} \quad\text{in }H^1((0,\infty),\mu_{m,\infty}),\quad \text{as}\ R\to \infty.
		\]
		
		\proofstep{3}{The limit of $H_m(R)$.}
		
		\noindent\hspace*{2em}Since $w_m'=(\frac{m}{r}-r)w_m$ and $z_{m,R}w_m\to0$ at the origin, integration by parts gives
		\[
		H_m(R)=m\E_{m,R}\frac{z_{m,R}}r -\frac{w_m(R)}{W_{m,R}}z_{m,R}(R).
		\]
		The constant extension from Step~2 also gives
		\[
		\E_{m,R}\frac{z_{m,R}}r =\frac{W_{m,\infty}}{W_{m,R}} \Big[ \Big\langle\frac1r,z_{m,R}\Big\rangle_{L^2(\mu_{m,\infty})} -z_{m,R}(R)\int_R^\infty\frac1r\,d\mu_{m,\infty} \Big].
		\]
		Here $W_{m,R}\to W_{m,\infty}$ as $R\to\infty$.
		The second term in brackets tends to zero because $z_{m,R}(R)$ is bounded and $\frac{1}{r}$ is integrable on the probability space; the first converges by the weak convergence in Step~2.
		Finally, $\frac{w_m(R)z_{m,R}(R)}{W_{m,R}}\to0$ as $R\to\infty$.
		Thus
		\[
		H_m(R)\longrightarrow m\Big\langle\frac1r,z_{m,\infty} \Big\rangle_{L^2(\mu_{m,\infty})}.
		\]
		
		\proofstep{4}{Compute the half-line value.}
		
		\noindent\hspace*{2em}Recall that $\nu=\frac{m+1}{2}$ and set
		\[
		F(s)=(2s)^{-\frac12}\quad\text{and}\quad Z(s)=z_{m,\infty}(\sqrt{2s}).
		\]
		Under $s=\frac{r^2}{2}$ the normalized measures satisfy
		\[
		d\mu_{m,\infty}(r) =\frac{s^{\nu-1}e^{-s}}{\Gamma(\nu)}\,ds =d\eta_\nu(s).
		\]
		Thus $u(r)\mapsto u(\sqrt{2s})$ is unitary from $L^2(\mu_{m,\infty})$ to $L^2(\eta_\nu)$, and hence
		\[
		\Big\langle\frac1r,z_{m,\infty} \Big\rangle_{L^2(\mu_{m,\infty})} =\langle F,Z\rangle_{L^2(\eta_\nu)}.
		\]
		Moreover, the weak equation from Step~1 becomes
		\[
		(A_\nu+2)Z=F.
		\]
		
		Write $F=\sum_{k\ge0}\widehat F_kL_k^{\nu-1}$.
		By Lemma~\ref{lem:standard-laguerre-facts},
		\[
		\widehat F_k=\frac{k!}{(\nu)_k}\langle F,L_k^{\nu-1}\rangle_{L^2(\eta_\nu)}=\frac1{\sqrt2}\frac{\Gamma(\nu-\frac12)}{\Gamma(\nu)}\frac{(\frac12)_k}{(\nu)_k}.
		\]
		Lemma~\ref{lem:laguerre-resolvent} then gives
		\[
		Z=\frac12\sum_{k=0}^\infty \frac{\widehat F_k}{k+1}L_k^{\nu-1}.
		\]
		Finally, orthogonality and $\|L_k^{\nu-1}\|_{L^2(\eta_\nu)}^2=\frac{(\nu)_k}{k!}$ give
		\[
		\Big\langle\frac1r,z_{m,\infty}\Big\rangle_{L^2(\mu_{m,\infty})}=\frac12\sum_{k=0}^\infty\frac{(\nu)_k}{k!}\frac{\widehat F_k^2}{k+1}=\frac14\Big(\frac{\Gamma(\nu-\frac12)}{\Gamma(\nu)}\Big)^2\sum_{k=0}^\infty\frac{(\frac12)_k^2}{(\nu)_k\,k!\,(k+1)}.
		\]
		
		To evaluate this series above, set
		\[
		F_\nu(t)=\sum_{k=0}^\infty \frac{(\frac12)_k^2}{(\nu)_k\,k!}t^k ={}_2F_1\Big(\frac12,\frac12;\nu;t\Big).
		\]
		Its coefficients are positive, so monotone convergence and $(k+1)^{-1}=\int_0^1t^k\,dt$ show that the series above equals $\int_0^1F_\nu(t)\,dt$.
		The hypergeometric differential equation \cite[Eq.~(15.10.1)]{NISTDLMF} gives
		\[
		F_\nu(t)=4\frac d{dt} \Bigl[t(1-t)F_\nu'(t)+(\nu-1)F_\nu(t)\Bigr].
		\]
		Since $\nu>1$, Gauss's formula \cite[Eq.~(15.4.20)]{NISTDLMF} gives
		\[
		F_\nu(1)= \frac{\Gamma(\nu)\Gamma(\nu-1)} {\Gamma(\nu-\frac12)^2}.
		\]
		
		The coefficients of $F_\nu$ are nonnegative and sum to this finite value, so $t(1-t)F_\nu'(t)\to0$ as $t\uparrow1$.
		Indeed, after splitting the derivative series at $N$, its finite part tends to zero and its tail is bounded by $\sum_{k>N}\frac{(\frac12)_k^2}{(\nu)_k k!}$, uniformly in $t$, because $k(1-t)t^k\le1$.
		Integrating the preceding derivative identity from $0$ to $1$ yields
		\[
		\int_0^1F_\nu(t)\,dt =4(\nu-1)\Big( \frac{\Gamma(\nu)\Gamma(\nu-1)} {\Gamma(\nu-\frac12)^2}-1\Big).
		\]
		Therefore
		\[
		\Big\langle\frac1r,z_{m,\infty} \Big\rangle_{L^2(\mu_{m,\infty})} =1-(\nu-1) \Big(\frac{\Gamma(\nu-\frac12)}{\Gamma(\nu)}\Big)^2.
		\]
		Combining Step~3 with the preceding identity gives
		\[
		\lim_{R\to\infty}\bigl(m-H_m(R)\bigr)
		=m(\nu-1)\Big(\frac{\Gamma(\nu-\frac12)}
		{\Gamma(\nu)}\Big)^2>0.
		\]
		Hence the denominator in \eqref{eq:c-lambda-formula}
		stays away from zero for all sufficiently large $R$.
		Using $m-1=2(\nu-1)$ in the same formula now gives
		\[
		\lim_{R\to\infty}\lambda_m(R) =1-\frac2m \frac{\Gamma(\nu)^2}{\Gamma(\nu-\frac12)^2},
		\]
		which is the asserted value.
	\end{proof}
	
	\subsection{The case \texorpdfstring{$m=1$}{m=1}}
	
	Now we turn to the case $m=1$.
	
	\begin{lemma}
		\label{lem:first-order-endpoint}
		Every locally absolutely continuous solution of $Tz=c$ on $(0,R)$ has the form
		\[
		z(r)=e^{\frac{r^2}{2}} \Bigl(C-c\int_0^r e^{-\frac{s^2}{2}}\,ds\Bigr).
		\]
		On $(0,\infty)$, finite energy in $H^1((0,\infty),w_1\,dr)$ forces $C=c\int_0^\infty e^{-\frac{s^2}{2}}\,ds$, and therefore
		\[
		z(r)=ce^{\frac{r^2}{2}}\int_r^\infty e^{-\frac{s^2}{2}}\,ds.
		\]
	\end{lemma}

	\begin{proof}
		Since $T=r-D$, the equation $Tz=c$ is $z'-rz=-c$.
		Hence
		\[
		\bigl(e^{-\frac{r^2}{2}}z(r)\bigr)'=-ce^{-\frac{r^2}{2}},
		\]
		and integration gives
		\[
		z(r)=e^{\frac{r^2}{2}}\Bigl(C-c\int_0^r e^{-\frac{s^2}{2}}\,ds\Bigr).
		\]
		This is
		\[
		z(r)=e^{\frac{r^2}{2}}\Bigl(C-c\int_0^\infty e^{-\frac{s^2}{2}}\,ds\Bigr)+ce^{\frac{r^2}{2}}\int_r^\infty e^{-\frac{s^2}{2}}\,ds.
		\]
		The second term is $O(r^{-1})$, since
		\[
		0\le e^{\frac{r^2}{2}}\int_r^\infty e^{-\frac{s^2}{2}}\,ds\le\frac1r, \quad \text{for}\ r>0.
		\]
		
		If the coefficient of the first term is nonzero, then $|z(r)|\ge a e^{\frac{r^2}{2}}$ for some $a>0$ and all sufficiently large $r$.
		Since $w_1(r)=re^{-\frac{r^2}{2}}$, this implies
		\[
		\int_0^\infty z(r)^2w_1(r)\,dr=\infty.
		\]
		Finite energy therefore forces $C=c\int_0^\infty e^{-\frac{s^2}{2}}\,ds$, which gives the asserted formula.
		For this choice, $z(r)=O(r^{-1})$ and $z'(r)=rz(r)-c=O(1)$ at infinity. These, together with the smoothness near $0$, give that $z\in H^1((0,\infty),w_1\,dr)$.
	\end{proof}

	For later reference, denote 
	\begin{equation}\label{eq:Lambda-one}
		\Lambda_1(R)=\inf_{\substack{v\in H^1((0,R),w_1\,dr)\\ Tv= \mathrm{constant}}}\Bigl\{\E_{1,R}(v'^2+v^2)+(1+Tv)^2\Bigr\}.
	\end{equation}
	
	\begin{lemma}
		\label{lem:m-one-continuity}
		For every fixed $R>0$,
		\[
		\Lambda_1(R)=\lim_{m\downarrow1}\lambda_m(R).
		\]
	\end{lemma}
	
	\begin{proof}
		Fix any sequence $m_j>1$ with $m_j\to1$.
		Since $W_{m_j,R}\to W_{1,R}$,
		\[
		\frac{d\mu_{m_j,R}}{d\mu_{1,R}}(r)=\frac{W_{1,R}}{W_{m_j,R}}r^{m_j-1}\longrightarrow1 \qquad (0<r<R).
		\]
		The convergence is uniform on every $[\eta,R]$, and the displayed ratio is bounded above on $(0,R)$ by a constant depending only on $R$ for all sufficiently large $j$.
		
		\proofstep{1}{The lower bound.}
		
		\noindent\hspace*{2em}Choose indices $j_k$ such that $\lambda_{m_{j_k}}(R)$ converges to $\liminf_{j\to\infty}\lambda_{m_j}(R)$.
		Let $v_k$ minimize $J_{m_{j_k},R}$ and set $\tau_k=\E_{m_{j_k},R}(Tv_k)$.
		Since $J_{m_{j_k},R}(v_k)\le J_{m_{j_k},R}(0)=\frac1{m_{j_k}}$ and the three terms in $J_{m_{j_k},R}$ are nonnegative,
		\[
		\E_{m_{j_k},R}(v_k'^2+v_k^2)\le\frac1{m_{j_k}},\qquad
		\Var_{m_{j_k},R}(Tv_k)\le\frac{m_{j_k}-1}{m_{j_k}}\quad\text{and}\quad
		|1+\tau_k|\le1.
		\]
		After passing to a further subsequence, $\tau_k\to\ell$.
		For every $\eta\in(0,R)$, the density comparison above makes $(v_k)$ bounded in $H^1(\eta,R)$.
		A diagonal extraction gives a function $v$ such that $v_k\rightharpoonup v$ in $H^1(\eta,R)$ for every $\eta>0$.
		The variance estimate and the positive lower bound for the densities on $[\eta,R]$ give
		\[
		Tv_k-\tau_k\longrightarrow0 \quad\text{in }L^2(\eta,R).
		\]
		On the other hand, weak $H^1$ convergence implies $Tv_k\rightharpoonup Tv$ in $L^2(\eta,R)$.
		Hence $Tv=\ell$ on $(0,R)$.
		Uniform convergence of the densities on $[\eta,R]$ and weak lower semicontinuity yield
		\[
		\liminf_{j\to\infty}\lambda_{m_j}(R)\ge\frac1{W_{1,R}}\int_\eta^R(v'^2+v^2)w_1\,dr+(1+\ell)^2.
		\]
		As $\eta\downarrow0$, the truncated energies increase and remain bounded by the finite left-hand side.
		Monotone convergence therefore gives $v\in H^1((0,R),w_1\,dr)$ and
		\[
		\liminf_{j\to\infty}\lambda_{m_j}(R)\ge\E_{1,R}(v'^2+v^2)+(1+\ell)^2\ge\Lambda_1(R).
		\]
		
		\proofstep{2}{The upper bound.}
		
		\noindent\hspace*{2em}Let $\widetilde v\in H^1((0,R),w_1\,dr)$ satisfy $T\widetilde v=a$ for some constant $a$.
		The density bound above makes $\widetilde v$ admissible for $J_{m_j,R}$ for all sufficiently large $j$.
		Since $\E_{m_j,R}(T\widetilde v)=a$ and $\Var_{m_j,R}(T\widetilde v)=0$, dominated convergence gives
		\[
		\begin{aligned}
			\limsup_{j\to\infty}\lambda_{m_j}(R)
			\le\lim_{j\to\infty}J_{m_j,R}(\widetilde v)=\E_{1,R}(\widetilde v'^2+\widetilde v^2)+(1+a)^2.
		\end{aligned}
		\]
		Taking the infimum over the constraint set in \eqref{eq:Lambda-one} gives $\limsup_{j\to\infty}\lambda_{m_j}(R)\le\Lambda_1(R)$.
		This, together with Step~1, proves the result, since the sequence $(m_j)$ was arbitrary.
	\end{proof}
	
	\begin{lemma}\label{lem:m-one}
		For every $R>0$,
		\begin{equation}\label{eq:m-one-dual}
			\lambda_1(R)=\Lambda_1(R)=\inf_{\substack{v\in H^1((0,R),w_1\,dr)\\ Tv\ \mathrm{constant}}}\Bigl\{\E_{1,R}(v'^2+v^2)+(1+Tv)^2\Bigr\}\ge1-\frac2\pi,
		\end{equation}
		and $\lambda_1(R)\downarrow1-\frac2\pi$ as $R\to\infty$.
	\end{lemma}
	
	\begin{proof}
		We divide the proof into four steps.
		
		\proofstep{1}{Reduction to the constrained problem.}
		
		\noindent\hspace*{2em}When $m=1$, the quadratic terms in $g$ cancel:
		\[
		\Phi_{1,R}(g,v) =\E_{1,R}(v'^2+v^2)+1+\E_{1,R}\big[(Tv)^2\big] +2\E_{1,R}(gTv).
		\]
		On the affine hyperplane $\E_{1,R}g=1$, the last term is bounded above precisely when $Tv$ is constant; in that case the supremum equals the expression on the right-hand side of \eqref{eq:m-one-dual}.
		The minimax inequality therefore gives $\lambda_1(R)\le\Lambda_1(R)$; the reverse inequality follows from the saddle verified in Step~2.
		
		\proofstep{2}{Construction and verification of the saddle.}
		
		\noindent\hspace*{2em}The affine set $\{z\in H^1((0,R),w_1\,dr):Tz=1\}$ is nonempty and closed.
		Since the energy $\E_{1,R}(z'^2+z^2)$ is coercive and strictly convex, it has a unique minimizer $z_{1,R}$.
		Set
		\[
		N_R=1+\E_{1,R}(z_{1,R}'^2+z_{1,R}^2), \quad \mathcal Z_{1,R}(r)=\frac1r\int_0^r s z_{1,R}(s)\,ds, \quad \mathcal Z_{1,R}(0)=0 \quad \text{and} \quad g_{1,R}=\frac{2-\mathcal Z_{1,R}}{N_R}.
		\]
		Then $Tz_{1,R}=1$ and Lemma~\ref{lem:first-order-endpoint} give, for some constant $C_{1,R}$,
		\[
		z_{1,R}(r)=e^{\frac{r^2}{2}} \Bigl(C_{1,R}-\int_0^r e^{-\frac{s^2}{2}}\,ds\Bigr).
		\]
		Hence $z_{1,R}$, $\mathcal Z_{1,R}$, and $g_{1,R}$ are bounded on $[0,R]$, so $g_{1,R}$ is admissible in \eqref{eq:lambda-definition}.
		
		Since $Tz_{1,R}=1$, the formal identities in \eqref{eq:adjoints} give, in the interior of $(0,R)$,
		\[
		(D^*D+2)z_{1,R}=\frac1r,\qquad T^*\mathcal Z_{1,R}=z_{1,R}.
		\]
		The feasible variations in this constrained minimization belong to $\ker T$, which is spanned by $e^{\frac{r^2}{2}}$.
		Stationarity in this direction, followed by integration by parts, gives
		\[
		w_1(R)e^{\frac{R^2}{2}} \bigl(z_{1,R}'(R)+1-\mathcal Z_{1,R}(R)\bigr)=0.
		\]
		Thus $z_{1,R}'(R)+1-\mathcal Z_{1,R}(R)=0$.
		A second integration by parts yields
		\begin{equation}\label{eq:m-one-abstract-bilinear}
			\E_{1,R}(z_{1,R}'h'+z_{1,R}h) =\E_{1,R}\big[(1-\mathcal Z_{1,R})Th\big],
		\end{equation}
		for every $h\in H^1((0,R),w_1\,dr)$.
		Taking $h=z_{1,R}$ shows that $N_R=2-\E_{1,R}\mathcal Z_{1,R}$, and hence $\E_{1,R}g_{1,R}=1$.
		
		Put $v_*=-\frac{z_{1,R}}{N_R}$.
		Identity \eqref{eq:m-one-abstract-bilinear} is the Euler equation of the strictly convex map $v\mapsto\Phi_{1,R}(g_{1,R},v)$ at $v_*$.
		Thus $v_*$ is its unique minimizer.
		Since $Tv_*=-\frac1{N_R}$ is constant, $\Phi_{1,R}(g,v_*)$ has the same value for every $g$ of mean one.
		These two saddle inequalities give
		\[
		\lambda_1(R)=\Phi_{1,R}(g_{1,R},v_*) =1-\frac1{N_R}.
		\]
		This also establishes the equality in \eqref{eq:m-one-dual}.
		
		\proofstep{3}{Lower bound and monotonicity of $N_R$.}
		
		\noindent\hspace*{2em}Combining Lemma~\ref{lem:m-one-continuity} with the equality $\lambda_1(R)=\Lambda_1(R)$ just proved gives
		\[
		1-\frac1{N_R} =\lambda_1(R)=\lim_{m\downarrow1}\lambda_m(R).
		\]
		Lemma~\ref{lem:finite-monotonicity} and Lemma~\ref{lem:half-line-value} now imply
		\[
		1-\frac1{N_R}\ge \lim_{m\downarrow1} \Biggl[1-\frac2m \frac{\Gamma(\frac{m+1}{2})^2}{\Gamma(\frac m2)^2}\Biggr] =1-\frac2\pi.
		\]
		Hence $N_R\ge\frac\pi2$, equivalently $\lambda_1(R)\ge1-\frac2\pi$.
		
		If $0<R_1<R_2$, Lemma~\ref{lem:finite-monotonicity} gives $\lambda_m(R_1)\ge\lambda_m(R_2)$ for every $m>1$.
		Passing to the limit $m\downarrow1$ yields
		\[
		1-\frac1{N_{R_1}}\ge1-\frac1{N_{R_2}}.
		\]
		Since $x\mapsto1-\frac1x$ is strictly increasing on $(0,\infty)$, $N_{R_1}\ge N_{R_2}$.
		Thus $R\mapsto N_R$ is nonincreasing.
		
		\proofstep{4}{Upper bound and limit of $N_R$.}
		
		\noindent\hspace*{2em}For the matching upper bound, use the Mills profile $M$ from Lemma~\ref{lem:first-order-endpoint}, which satisfies $TM=1$.
		Since $M'=rM-1$, Fubini's theorem and one integration by parts give
		\[
		\int_0^\infty M^2r e^{-\frac{r^2}{2}}dr =2-\frac\pi2, \qquad \int_0^\infty M e^{-\frac{r^2}{2}}dr =1.
		\]
		Differentiating $M'=rM-1$ gives the pointwise differential identity
		\[
		-M''+\Bigl(r-\frac1r\Bigr)M'+2M=\frac1r, \qquad r>0.
		\]
		This is an identity for the differential expression, not an equation in the domain of the half-line self-adjoint operator. Indeed, $\frac1r\notin L^2((0,\infty),w_1\,dr)$.
		Multiply the identity by $M$ on a compact subinterval and integrate by parts.
		The explicit Mills formula gives $M(0)<\infty$, $M'(0)=-1$, and $M(r)=r^{-1}+O(r^{-3})$ as $r\to\infty$, so the boundary terms vanish as the subinterval exhausts $(0,\infty)$.
		We obtain
		\begin{equation}\label{eq:Mills-energy}
			\int_0^\infty(M'^2+M^2)r e^{-\frac{r^2}{2}}\,dr=\frac\pi2-1.
		\end{equation}
		The restriction of $M$ to $(0,R)$ still satisfies $TM=1$ exactly.
		It is therefore admissible in the constrained minimization defining $z_{1,R}$.
		By the minimality of $z_{1,R}$ and \eqref{eq:Mills-energy},
		\[
		N_R =1+\E_{1,R}(z_{1,R}'{}^2+z_{1,R}^2) \le1+\E_{1,R}(M'^2+M^2),
		\]
		while \eqref{eq:Mills-energy} and $W_{1,R}\to W_{1,\infty}=1$ give
		\[
		1+\E_{1,R}(M'^2+M^2)\longrightarrow\frac\pi2.
		\]
		Hence $\limsup_{R\to\infty}N_R\le\frac\pi2$.
		Together with Step~3, this proves
		\[
		N_R\downarrow\frac\pi2.
		\]
		Finally,
		\[
		\lambda_1(R)=1-\frac1{N_R} \downarrow1-\frac2\pi,
		\]
		as asserted.
	\end{proof}
	
	\begin{proof}[Proof of Theorem~\ref{thm:radial-value}]
		For $m>1$, Lemma~\ref{lem:finite-monotonicity} and Lemma~\ref{lem:half-line-value} give $\lambda_m(R)\ge\kappa_m$ and \eqref{eq:radial-limit}.
		The endpoint is Lemma~\ref{lem:m-one}.
		The attainment assertion follows from Lemma~\ref{lem:saddle-reduction} when $m>1$ and from the saddle construction in the proof of Lemma~\ref{lem:m-one} when $m=1$.
	\end{proof}
	
	This completes the one-dimensional argument needed for the geometric lifting.
	The closed forms of the maximizing right-hand sides play no role below and are recorded separately in Appendix~\ref{app:explicit-profiles} for completeness.
	
	\section{\bf The lower bound by raywise lifting}\label{lift}
	
	In this part, we give the required lower bound. We first isolate the radial part of the Hessian, then verify that the one-dimensional optimizing profiles can be selected measurably, and finally assemble those profiles along the rays of the convex body.
	
	\begin{lemma}
		\label{lem:polar-hessian}
		Let $\Omega\subset\R^n$ be open, $m=n-1$,  $u\in H^2_{\mathrm{loc}}(\Omega)$, and suppose that $Lu=g$ almost everywhere in $\Omega$.
		At $x=r\theta\in\Omega\setminus\{0\}$, let
		\[
		u_r(r\theta)=\partial_r[u(r\theta)] \quad\text{and}\quad u_{rr}(r\theta)=\partial_r^2[u(r\theta)]
		\]
		denote the radial derivatives.
		Then, almost everywhere,
		\[
		\|\Hess u\|_{\HS}^2+|\nabla u|^2 \ge u_{rr}^2+u_r^2+\frac{(g+ru_r-u_{rr})^2}{m}.
		\]
	\end{lemma}
	
	\begin{proof}
		
		Complete $\theta$ to an orthonormal frame $\theta,e_1,\ldots,e_m$, with the last $m$ vectors tangent to the sphere.
		The radial entry of the Hessian is $u_{rr}$.
		Since $Lu=\Delta u-x\cdot\nabla u=g$ and $x\cdot\nabla u=ru_r$,
		\[
		\sum_{i=1}^m\Hess u(e_i,e_i) =\Delta u-u_{rr}=g+ru_r-u_{rr}.
		\]
		Cauchy--Schwarz applied only to these $m$ diagonal tangential entries gives
		\[
		\|\Hess u\|_{\HS}^2 \ge u_{rr}^2+\frac{(g+ru_r-u_{rr})^2}{m}.
		\]
		Since $|\nabla u|^2\ge u_r^2$, this proves the claim.
	\end{proof}
	
	For $K\in\mathcal K_{+,o}^n$, let $\rho_K$ be the radial function defined in Section~\ref{sec:prelim}.
	For every $R>0$, choose a maximizer $g_{m,R}$ supplied by Theorem~\ref{thm:radial-value}.
	Since it is an admissible maximizer, $\E_{m,R}g_{m,R}=1$ and, for every $v\in H^1((0,R),w_m\,dr)$,
	\begin{equation}\label{eq:profile-inequality}
		\E_{m,R}\Bigl[v'^2+v^2+\frac{(g_{m,R}+Tv)^2}{m}-(g_{m,R}-1)^2\Bigr]=\Phi_{m,R}(g_{m,R},v)\ge\lambda_m(R)\ge\kappa_m.
	\end{equation}
	No closed-form expression for $g_{m,R}$ is needed below. The next lemma provides the measurability and local $L^2$ bounds required for the raywise construction.
	
	We record the parameter dependence needed to assemble these profiles.
	
	\begin{lemma}
		\label{lem:profile-dependence}
		Fix $m\ge1$ and $0<a<b<\infty$.
		The profiles may be chosen so that
		\begin{equation}\label{eq:profile-L2-continuity}
			R\longmapsto g_{m,R}(R\,\cdot)
		\end{equation}
		is continuous from $[a,b]$ into $L^2((0,1),s^m\,ds)$.
		Moreover, there is a Borel function
		\[
		G:[a,b]\times(0,1)\longrightarrow\R.
		\]
		For every $R$, the function $G(R,\cdot)$ represents $g_{m,R}(R\,\cdot)$.
		Consequently, the profiles have representatives for which $(R,r)\mapsto g_{m,R}(r)$ is Borel measurable when $a\le R\le b$ and $0<r<R$.
		Finally,
		\begin{equation}\label{eq:profile-uniform-L2}
			\sup_{a\le R\le b}\E_{m,R}g_{m,R}^2<\infty.
		\end{equation}
	\end{lemma}
	
	\begin{proof}
		\proofstep{1}{Reduction to a fixed interval.}
		
		\noindent\hspace*{2em}Rescale by $r=Rs$.
		For $R\in[a,b]$,
		\[
		d\mu_{m,R}(Rs) =\frac{R^{m+1}s^m e^{-\frac{R^2s^2}{2}}}{W_{m,R}}\,ds, \qquad T v(Rs)=RsV(s)-\frac1R V'(s),\quad V(s)=v(Rs).
		\]
		Thus the pulled-back measures are uniformly equivalent to $s^m\,ds$, and their densities depend continuously on $R$.
		
		\proofstep{2}{Continuity of the pulled-back profiles.}
		
		\noindent\hspace*{2em}Suppose first that $m>1$.
		Pulling \eqref{eq:J-functional} back to $(0,1)$ gives strictly convex quadratic functionals whose bilinear parts are uniformly coercive and whose bilinear and linear parts depend continuously on $R$ in operator norm on $H^1((0,1),s^m\,ds)$.
		The parameter-dependence argument in the proof of Lemma~\ref{lem:envelope} shows that their unique minimizers depend continuously on $R$.
		Formula \eqref{eq:g-fixed-v} then shows that $R\mapsto g_{m,R}(R\,\cdot)$ is continuous in $L^2((0,1),s^m\,ds)$.
		
		For $m=1$, Lemma~\ref{lem:first-order-endpoint} gives
		\[
		z(Rs)=e^{\frac{R^2s^2}{2}} \Bigl(C-\int_0^{Rs}e^{-\frac{t^2}{2}}\,dt\Bigr).
		\]
		Its energy is a strictly convex quadratic polynomial in $C$, whose coefficients depend continuously on $R$.
		Hence the minimizing coefficient $C_{1,R}$, and therefore $z_{1,R}(R\,\cdot)$, depend continuously on $R$.
		The same is true of
		\[
		\mathcal Z_{1,R}(Rs) =\frac R s\int_0^s t z_{1,R}(Rt)\,dt, \qquad g_{1,R}(Rs)=\frac{2-\mathcal Z_{1,R}(Rs)}{N_R},
		\]
		with the value at $s=0$ understood by continuity. Here $N_R\ge1$ and $N_R$ is continuous.
		This proves \eqref{eq:profile-L2-continuity} also at the endpoint.
		
		\proofstep{3}{The uniform $L^2$ bound.}
		
		\noindent\hspace*{2em}By Step~2 and compactness of $[a,b]$,
		\[
		\sup_{R\in[a,b]} \|g_{m,R}(R\,\cdot)\|_{L^2((0,1),s^m\,ds)}<\infty.
		\]
		The uniform equivalence of the pulled-back measures from Step~1 now gives \eqref{eq:profile-uniform-L2}.
		
		\proofstep{4}{A jointly measurable realization.}
		
		\noindent\hspace*{2em}The source constructed in Theorem~\ref{thm:sharp-local} uses the profile with $R=\rho_K(\theta)$.
		Since Step~2 gives only an $L^2$-equivalence class for each $R$, we must choose actual functions simultaneously so that their values are measurable in both $R$ and $r$.
		
		To distinguish the $L^2$-class from its pointwise representative, write
		\[
		\mathcal F(R) =\bigl[s\longmapsto g_{m,R}(Rs)\bigr] \in L^2((0,1),s^m\,ds).
		\]
		By Step~2, $\mathcal F:[a,b]\to L^2((0,1),s^m\,ds)$ is continuous.
		Hence $\mathcal F([a,b])$ is compact.
		
		For every $k\ge1$, choose a finite $4^{-k}$-net of this compact set and assign each $\mathcal F(R)$ to the first net point within distance $4^{-k}$.
		Since the distance from $\mathcal F(R)$ to each net point is continuous in $R$, this defines a finite-valued Borel map $\mathcal F_k:[a,b]\to L^2((0,1),s^m\,ds)$ such that
		\[
		\sup_{R\in[a,b]} \|\mathcal F_k(R)-\mathcal F(R)\|_2\le4^{-k}.
		\]
		Choosing a Borel representative of each of the finitely many values of $\mathcal F_k$ gives a Borel function $G_k(R,s)$ representing $\mathcal F_k(R)$ for every $R$.
		
		All norms in the following display are taken with respect to $s^m\,ds$.
		Since this measure is finite, for every fixed $R$,
		\[
		\sum_{k=1}^\infty \|G_{k+1}(R,\cdot)-G_k(R,\cdot)\|_1 \le \frac1{\sqrt{m+1}} \sum_{k=1}^\infty(4^{-k-1}+4^{-k}) <\infty.
		\]
		By Tonelli's theorem, $G_k(R,s)$ converges for almost every $s$.
		Let $G(R,s)$ be this limit, setting it equal to zero where the limit does not exist.
		The convergence set is Borel, so $G$ is a Borel function on $[a,b]\times(0,1)$.
		Moreover, $G_k(R,\cdot)\to G(R,\cdot)$ in $L^1$, while $\mathcal F_k(R)\to\mathcal F(R)$ in $L^2$ and hence in $L^1$.
		Uniqueness of the $L^1$ limit shows that $G(R,\cdot)$ represents $\mathcal F(R)$ for every $R$.
		
		Setting
		\[
		g_{m,R}(r)=G\bigl(R,\frac rR\bigr),\qquad 0<r<R,
		\]
		gives the required representatives, and they are Borel measurable in $(R,r)$ because $(R,r)\mapsto\bigl(R,\frac rR\bigr)$ is continuous.
	\end{proof}
	
	We can now assemble the profiles along the rays of $K$.
	
	\begin{theorem}\label{thm:sharp-local}
		Let $K\in\mathcal K_{+,o}^n$, $m=n-1$ and $f\in C^1(\partial K)$.
		There is a right-hand side $g\in L^2(K,\gamma_n)$ satisfying the compatibility condition \eqref{eq:compatibility} such that the mean-zero solution $u\in H^2(K)$ of the Gaussian Neumann problem satisfies
		\begin{equation}\label{eq:sharp-energy-K}
			\int_K\bigl(\|\Hess u\|_{\HS}^2+|\nabla u|^2\bigr)d\gamma_n -\int_K(g-\bar g_K)^2\,d\gamma_n \ge\kappa_m\gamma_n(K)\bar g_K^2.
		\end{equation}
		Here $\bar g_K=\gamma_n(K)^{-1}\int_Kg\,d\gamma_n$.
	\end{theorem}
	
	\begin{proof}
		Set
		\[
		s_f=\frac1{\gamma_n(K)} \int_{\partial K}f\,d\gamma_{\partial K},
		\]
		and, using the measurable representatives supplied by Lemma~\ref{lem:profile-dependence}, define the desired source $g_K$ in polar coordinates by
		\[
		g_K(r\theta)=s_f\, g_{m,\rho_K(\theta)}(r), \qquad 0<r<\rho_K(\theta).
		\]
		Since $\rho_K$ is continuous, the measurability established in Lemma~\ref{lem:profile-dependence} shows that $g_K$ is measurable.
		Polar coordinates also give
		\[
		\gamma_n(K)=(2\pi)^{-\frac n2} \int_{\Sph^{n-1}}W_{m,\rho_K(\theta)}\,d\theta.
		\]
		Consequently, the normalization $\E_{m,R}g_{m,R}=1$ gives
		\[
		\int_K g_K\,d\gamma_n=(2\pi)^{-\frac n2}s_f\int_{\Sph^{n-1}}W_{m,\rho_K(\theta)}\,d\theta=s_f\gamma_n(K)=\int_{\partial K}f\,d\gamma_{\partial K}.
		\]
		Thus $g_K$ is compatible with the Neumann datum $f$, and $\bar g_K=s_f$.
		
		Since $0\in\operatorname{int}K$, there are $0<a\le b<\infty$ such that $a\le\rho_K(\theta)\le b$.
		Lemma~\ref{lem:profile-dependence} gives
		\[
		\int_K g_K^2\,d\gamma_n=(2\pi)^{-\frac n2}s_f^2 \int_{\Sph^{n-1}}W_{m,\rho_K(\theta)} \E_{m,\rho_K(\theta)}g_{m,\rho_K(\theta)}^2\,d\theta <\infty.
		\]
		
		Let $u\in H^2(K)$ be the Gaussian mean-zero solution of
		\[
		Lu=g_K\quad\text{in }K,\qquad u_{\nu_K}=f\quad\text{on }\partial K.
		\]
		Its existence and regularity follow from Lemma~\ref{lem:gaussian-neumann}.
		Since $u\in H^2(K)$, the bounds $|u_r|\le|\nabla u|$ and $|u_{rr}|\le\|\Hess u\|_{\HS}$, together with polar coordinates and Fubini's theorem, show that $u_r(\,\cdot\,\theta)$ and $u_{rr}(\,\cdot\,\theta)$ belong to $L^2((0,\rho_K(\theta)),w_m\,dr)$ for almost every $\theta$.
		On $K\setminus\{0\}$, the identity $u_r=\frac{x}{|x|}\cdot\nabla u$ shows that $u_r\in H^1_{\mathrm{loc}}$.
		Sobolev slicing on a countable exhaustion of $K\setminus\{0\}$ by annuli identifies the weak radial derivative of $u_r$ with $u_{rr}$.
		Hence, by the definition in Subsection~\ref{ssec:functional-tools}, for almost every $\theta$,
		\[
		r\longmapsto u_r(r\theta) \in H^1((0,\rho_K(\theta)),w_m\,dr), \qquad \partial_r[u_r(r\theta)]=u_{rr}(r\theta).
		\]
		For such a direction, the one-dimensional operator in \eqref{eq:profile-inequality} acts along the ray:
		\[
		T\bigl(u_r(\,\cdot\,\theta)\bigr)(r) =ru_r(r\theta)-\partial_r[u_r(r\theta)] =ru_r(r\theta)-u_{rr}(r\theta).
		\]
		If $s_f\ne0$, apply \eqref{eq:profile-inequality} to $r\mapsto\frac{u_r(r\theta)}{s_f}$ and multiply by $s_f^2$; if $s_f=0$, then $g_K=0$ and the same raywise lower bound follows from the nonnegativity of the remaining quadratic energy.
		Lemma~\ref{lem:polar-hessian} gives the first inequality below, while the raywise profile inequality gives the second:
		\begin{align*}
			&\int_K\bigl(\|\Hess u\|_{\HS}^2+|\nabla u|^2\bigr)d\gamma_n-\int_K(g_K-s_f)^2\,d\gamma_n\\
			\ge{}&(2\pi)^{-\frac n2}\int_{\Sph^{n-1}}\int_0^{\rho_K(\theta)}\Bigl[u_{rr}^2+u_r^2+\frac{(g_K+ru_r-u_{rr})^2}{m}-(g_K-s_f)^2\Bigr]w_m(r)\,dr\,d\theta\\
			\ge{}&(2\pi)^{-\frac n2}\kappa_m s_f^2\int_{\Sph^{n-1}}W_{m,\rho_K(\theta)}\,d\theta\\
			={}&\kappa_m\gamma_n(K)s_f^2=\kappa_m\gamma_n(K)\bar g_K^2.
		\end{align*}
		This is \eqref{eq:sharp-energy-K}.
	\end{proof}
	
	\begin{proof}[Proof of Theorem~\ref{thm:main}]
		We prove the lower bound first for smooth strictly convex bodies and then for arbitrary convex bodies by Hausdorff approximation. Lemma~\ref{lem:sharpness} gives the reverse bound and completes the proof.
		
		\proofstep{1}{The lower bound for smooth bodies.}
		
		\noindent\hspace*{2em}Set $m=n-1$.
		For $X\sim\chi_m$, the identities $\E X=\sqrt2\,\frac{\Gamma(\frac{m+1}{2})}{\Gamma(\frac m2)}$ and $\E X^2=m$ give $q_n=\kappa_m=\frac{\Var(X)}m>0$.
		Theorem~\ref{thm:sharp-local} and Lemma~\ref{lem:nonconstant-criterion}, applied with $q=q_n$, prove the desired inequality for $K,L\in\mathcal K_{+,o}^n$.
		
		\proofstep{2}{Pass to arbitrary convex bodies.}
		
		\noindent\hspace*{2em}Let $K,L\in\mathcal K_o^n$.
		By the standard smoothing of support functions \cite{Sch}, choose $K_j,L_j\in\mathcal K_{+,o}^n$ converging to $K,L$, respectively, in the Hausdorff metric.
		Such approximations may be obtained by first adding $j^{-1}B_2^n$ and then smoothing the support function by an approximate identity on the rotation group; the added ball keeps the support function positive and its curvature matrix positive definite.
		For each $t\in[0,1]$, continuity of Minkowski addition gives
		\[
		(1-t)K_j+tL_j\longrightarrow(1-t)K+tL
		\]
		in the Hausdorff metric.
		The boundary of a convex body has Gaussian measure zero, so Hausdorff convergence of convex bodies implies convergence of their Gaussian measures.
		We may therefore pass to the limit in the inequality from Step~1 and obtain it for $K,L$.
		This proves $\alpha_\gamma(n)\ge q_n$.
		
		\proofstep{3}{The reverse bound.}
		
		\noindent\hspace*{2em}Lemma~\ref{lem:sharpness} gives $\alpha_\gamma(n)\le q_n$.
		This, together with Step~2, gives $\alpha_\gamma(n)=q_n$.
	\end{proof}
	
	\appendix
	
	\newpage
	
	\section{\bf Explicit optimal right-hand sides}
	\label{app:explicit-profiles}
	
	The saddle construction in Section~\ref{sec:one-dim} already proves the main result.
	Here we combine it with the Laguerre resolution from Subsection~\ref{ssec:radial-tools} only to record closed formulas for the maximizing profiles.
	
	\subsection{The profile for \texorpdfstring{$m>1$}{m greater than 1}}
	
	For $m>1$, the finite-interval solution is obtained by adding a regular homogeneous correction to the half-line resolvent.
	
	\begin{lemma}
		Let $m>1$, $R>0$, and $\nu=\frac{m+1}{2}$.
		\[
		z_{m,\infty}(r)=\frac{\Gamma(\nu-\frac12)}{2\sqrt2\,\Gamma(\nu)}
		\sum_{k=0}^\infty\frac{(\frac12)_k}{(\nu)_k(k+1)}L_k^{\nu-1}\Bigl(\frac{r^2}{2}\Bigr),
		\qquad
		\phi_m(r)={}_1F_1\Bigl(1;\nu;\frac{r^2}{2}\Bigr).
		\]
		The boundary-correction coefficient
		\[
		C_{m,R}=\frac{Rz_{m,\infty}(R)-1-mz_{m,\infty}'(R)}{m\phi_m'(R)-R\phi_m(R)}
		\]
		is well defined, and $z_{m,R}=z_{m,\infty}+C_{m,R}\phi_m$ is the solution of \eqref{eq:z-ode}--\eqref{eq:z-boundary} on the regular branch at the origin.
		For $r>0$, define
		\[
		\mathcal Z_{m,R}(r)=r^{-m}\int_0^r s^m z_{m,R}(s)\,ds,
		\qquad \mathcal Z_{m,R}(0)=0.
		\]
		Then
		\[
		\begin{aligned}
			\mathcal Z_{m,R}(r)={}&\frac{\Gamma(\nu-\frac12)}{4\sqrt2\,\Gamma(\nu)}\,r
			\sum_{k=0}^\infty\frac{(\frac12)_k}{(\nu)_k(k+1)(k+\nu)}L_k^\nu\Bigl(\frac{r^2}{2}\Bigr)\\
			&+C_{m,R}\frac r{2\nu}{}_1F_1\Bigl(1;\nu+1;\frac{r^2}{2}\Bigr).
		\end{aligned}
		\]
		The maximizing profile is
		\begin{equation}\label{eq:g-m-explicit}
			g_{m,R}(r)=\frac{\dfrac{m+1}{m}-\mathcal Z_{m,R}(r)}{\dfrac{m+1}{m}-\E_{m,R}\mathcal Z_{m,R}},\qquad 0\le r\le R.
		\end{equation}
	\end{lemma}
	
	\begin{proof}
		We divide the proof into four steps.
		
		\proofstep{1}{The half-line particular solution.}

		Under the unitary map $U$, the function $r^{-1}$ becomes $F(s)=(2s)^{-\frac12}$.
		Lemma~\ref{lem:standard-laguerre-facts} gives its Laguerre coefficients
		\[
		\widehat F_k=\frac{\Gamma(\nu-\frac12)}{\sqrt2\,\Gamma(\nu)}\frac{(\frac12)_k}{(\nu)_k}.
		\]
		Lemma~\ref{lem:laguerre-resolvent} therefore gives the displayed series for $z_{m,\infty}=(D^*D+2)^{-1}(r^{-1})$.
		The series converges in the graph norm of $D^*D$ and in $H^2$ on compact subintervals of $(0,\infty)$, so it may be differentiated there term by term.
		It is also absolutely convergent at $r=0$ because $L_k^{\nu-1}(0)=\frac{(\nu)_k}{k!}$.
		
		\proofstep{2}{The regular correction.}

		Under $s=\frac{r^2}{2}$, the homogeneous equation associated with $D^*D+2$ becomes Kummer's equation
		\[
		s\phi''+(\nu-s)\phi'-\phi=0.
		\]
		Its normalized regular solution is $\phi_m(r)={}_1F_1(1;\nu;\frac{r^2}{2})$; see \cite[\S\S13.2(i) and 13.4(i)]{NISTDLMF}.
		Euler's integral representation gives
		\[
		\phi_m(r)=(\nu-1)\int_0^1e^{\frac{r^2t}{2}}(1-t)^{\nu-2}\,dt
		\]
		and hence
		\[
		\frac{\phi_m'(R)}{R\phi_m(R)}
		=\frac{\int_0^1t e^{\frac{R^2t}{2}}(1-t)^{\nu-2}\,dt}{\int_0^1e^{\frac{R^2t}{2}}(1-t)^{\nu-2}\,dt}
		>\frac1\nu>\frac1m.
		\]
		The first inequality follows because the exponential tilt strictly increases the mean of the beta density.
		Thus $m\phi_m'(R)-R\phi_m(R)>0$.
		Imposing \eqref{eq:z-boundary} on $z_{m,\infty}+C\phi_m$ gives $C=C_{m,R}$, and uniqueness in Lemma~\ref{lem:saddle-reduction} identifies the resulting function with $z_{m,R}$.
		
		\proofstep{3}{The radial primitive.}

		The identities
		\begin{align*}
			\int_0^r s^m L_k^{\nu-1}\Bigl(\frac{s^2}{2}\Bigr)\,ds&=\frac{r^{m+1}}{2(k+\nu)}L_k^\nu\Bigl(\frac{r^2}{2}\Bigr),\\
			\int_0^r s^m {}_1F_1\Bigl(1;\nu;\frac{s^2}{2}\Bigr)\,ds&=\frac{r^{m+1}}{2\nu}{}_1F_1\Bigl(1;\nu+1;\frac{r^2}{2}\Bigr)
		\end{align*}
		give the displayed series for $\mathcal Z_{m,R}$; the passage from the Laguerre partial sums to the integral follows from weighted $L^2$ convergence.
		A direct calculation using \eqref{eq:z-ode} gives
		\[
		\Biggl[r^m\Bigl(Tz_{m,R}-\frac1m\Bigr)\Biggr]' =(m-1)r^m z_{m,R}.
		\]
		The regularity at the origin removes the lower boundary term, so integration gives
		\[
		Tz_{m,R}(r)=\frac1m+(m-1)\mathcal Z_{m,R}(r).
		\]
		
		\proofstep{4}{Recovery of the optimizer.}

		Substituting the last identity into \eqref{eq:g-optimal} gives \eqref{eq:g-m-explicit}.
		Its denominator equals $\frac{m-H_m(R)}{m-1}>0$, and it is the expectation of the numerator, so $\E_{m,R}g_{m,R}=1$.
		Finally, the integral definition of $\mathcal Z_{m,R}$ gives $\mathcal Z_{m,R}(r)=O(r)$ at the origin.
	\end{proof}
	
	\subsection{The profile for \texorpdfstring{$m=1$}{m equal to 1}}
	
	For $m=1$, the constraint $Tz=1$ can be solved directly, so no Laguerre expansion is needed.
	
	\begin{lemma}
		Fix $R>0$ and put
		\[
		\Psi(r)=\int_0^r e^{-\frac{s^2}{2}}\,ds, \qquad \Psi_R=\Psi(R),
		\]
		\[
		C_{1,R}=\frac{e^{\frac{R^2}{2}}(R^2-1)\Psi_R+R}{e^{\frac{R^2}{2}}(R^2-1)+1},
		\qquad z_{1,R}(r)=e^{\frac{r^2}{2}}\bigl(C_{1,R}-\Psi(r)\bigr).
		\]
		For $r>0$, define
		\[
		\mathcal Z_{1,R}(r)=\frac1r\int_0^r s z_{1,R}(s)\,ds,
		\qquad \mathcal Z_{1,R}(0)=0.
		\]
		Then
		\[
		\mathcal Z_{1,R}(r)=1+\frac{z_{1,R}(r)-C_{1,R}}r
		\]
		and
		\[
		N_R:=1+\E_{1,R}(z_{1,R}'{}^2+z_{1,R}^2)
		=2-\E_{1,R}\mathcal Z_{1,R}
		=\frac{R\Psi_R+C_{1,R}(\Psi_R-R)}{W_{1,R}}>0.
		\]
		The optimal profile is
		\begin{equation}\label{eq:g-one-simplified}
			g_{1,R}(r)=\frac{1+\dfrac{C_{1,R}-z_{1,R}(r)}{r}}{N_R}\quad(0<r\le R),\qquad g_{1,R}(0)=\frac2{N_R}.
		\end{equation}
		Moreover, $\E_{1,R}g_{1,R}=1$.
	\end{lemma}
	
	\begin{proof}
		We divide the proof into three steps.
		
		\proofstep{1}{The constrained minimizer.}

		Lemma~\ref{lem:first-order-endpoint} gives every solution of $Tz=1$ in the form $z(r)=e^{\frac{r^2}{2}}(C-\Psi(r))$.
		Its energy is a strictly convex quadratic function of $C$, and its critical point satisfies
		\[
		0=\int_0^R r(rz_{1,R}'+z_{1,R})\,dr =R^2z_{1,R}(R)-\int_0^R r z_{1,R}(r)\,dr.
		\]
		Substitution of the preceding form of $z$ gives the stated value of $C_{1,R}$.
		Its denominator is
		\[
		e^{\frac{R^2}{2}}(R^2-1)+1=\int_0^R r(r^2+1)e^{\frac{r^2}{2}}\,dr>0,
		\]
		so this critical point is the unique constrained minimizer.
		
		\proofstep{2}{The radial primitive and normalization.}

		Since $z_{1,R}'=rz_{1,R}-1$, integration gives
		\[
		\mathcal Z_{1,R}(r)=1+\frac{z_{1,R}(r)-C_{1,R}}r.
		\]
		This also shows that $\mathcal Z_{1,R}(r)\to0$ as $r\downarrow0$.
		Lemma~\ref{lem:m-one} gives $N_R=2-\E_{1,R}\mathcal Z_{1,R}$, and direct substitution yields
		\[
		N_R=\frac{R\Psi_R+C_{1,R}(\Psi_R-R)}{W_{1,R}}.
		\]
		
		\proofstep{3}{Recovery of the optimizer.}

		Lemma~\ref{lem:m-one} gives $g_{1,R}=(2-\mathcal Z_{1,R})/N_R$.
		The formula for $\mathcal Z_{1,R}$ now gives \eqref{eq:g-one-simplified}, including the value at the origin.
		Finally, $N_R=2-\E_{1,R}\mathcal Z_{1,R}$ gives $\E_{1,R}g_{1,R}=1$.
	\end{proof}

\end{document}